\documentclass[11pt,reqno]{amsart}

\usepackage{color}
\usepackage{amsmath, amsthm, amssymb}
\usepackage{amsfonts}

\usepackage{hyperref}

\usepackage[ansinew]{inputenc}
\usepackage[dvips]{epsfig}
\usepackage{graphicx}
\usepackage[english]{babel}

\usepackage{thmtools}
\theoremstyle{plain}
\declaretheorem[title=Theorem, parent=section]{theorem}
\declaretheorem[title=Lemma,sibling=theorem]{lemma}
\declaretheorem[title=Proposition,sibling=theorem]{proposition}
\declaretheorem[title=Corollary,sibling=theorem]{corollary}

\theoremstyle{definition}
\declaretheorem[title=Definition,sibling=theorem]{definition}
\declaretheorem[title=Remark,sibling=theorem]{remark}
\declaretheorem[title=Remark, numbered=no]{remark*}

\declaretheorem[title=Assumption, numbered=no]{assumption*}

\numberwithin{equation}{section}

\usepackage{dsfont}
\usepackage{bbm}

\newcommand{\N}{\mathbb{N}}
\newcommand{\R}{\mathbb{R}}

\newcommand{\cE}{\mathcal{E}}

\newcommand{\eps}{\varepsilon}

\DeclareMathOperator{\dist}{dist}

\DeclareMathOperator{\supp}{supp}

\renewcommand{\d}{\textnormal{d}}
\DeclareMathOperator{\tail}{Tail}

\newcommand{\average}{{\mathchoice {\kern1ex\vcenter{\hrule height.4pt
width 6pt depth0pt} \kern-9.7pt} {\kern1ex\vcenter{\hrule
height.4pt width 4.3pt depth0pt} \kern-7pt} {} {} }}
\newcommand{\dashint}{\average\int}

\begin{document}

\title[Obstacle problems for nonlocal operators with singular kernels]{Obstacle problems for nonlocal operators \\ with singular kernels}

\author{Xavier Ros-Oton}
\author{Marvin Weidner}

\address{ICREA, Pg. Llu\'is Companys 23, 08010 Barcelona, Spain \& Universitat de Barcelona, Departament de Matem\`atiques i Inform\`atica, Gran Via de les Corts Catalanes 585, 08007 Barcelona, Spain \& Centre de Recerca Matem\`atica, Barcelona, Spain}
\email{xros@icrea.cat}

\address{Universitat de Barcelona, Departament de Matem\`atiques i Inform\`atica, Gran Via de les Corts Catalanes 585, 08007 Barcelona, Spain}
\email{mweidner@ub.edu}




\begin{abstract}
In this paper we establish optimal regularity estimates and smoothness of free boundaries for nonlocal obstacle problems governed by a very general class of integro-differential operators with possibly singular kernels. 
More precisely, in contrast to all previous known results, we are able to treat nonlocal operators whose kernels are not necessarily pointwise comparable to the one of the fractional Laplacian. 
Such operators might be very anisotropic in the sense that they ``do not see'' certain directions at all, or might have substantial oscillatory behavior, causing the nonlocal Harnack inequality to fail.
\end{abstract}

\maketitle

\section{Introduction}  

The goal of this article is to study the regularity theory for nonlocal obstacle problems
\begin{align}
\label{eq:OP}
\min\{ L u, u - \phi \} = 0~~ \text{ in } \R^n,
\end{align}
where $\phi : \R^n \to \R$ is a smooth obstacle and $L$ is an integro-differential operator of the form
\begin{align}
\label{eq:L}
L u(x) = \text{p.v.} \int_{\R^n} \big(u(x) - u(x+y)\big) K(y) \d y.
\end{align}
The operator $L$ is governed by its jumping kernel $K : \R^n \to [0,\infty]$ which is assumed to be symmetric, i.e. $K(y) = K(-y)$, and homogeneous of degree $-n-2s$, where $s \in (0,1)$, i.e.
\begin{align}
\label{eq:hom}
K(y) = |y|^{-n-2s}K(y/|y|).
\end{align} 
Due to the homogeneity of $K$, the operator $L$ is a stable operator of order $2s$, whose anisotropy is governed by the action of $K$ on the unit sphere $\mathbb{S}^{n-1}$.

\subsection{Background}
Obstacle problems of the type \eqref{eq:OP}-\eqref{eq:L} arise naturally in probability and mathematical finance (optimal stopping for L\'evy processes, pricing of options), as well as in models of interacting energies in physical, biological, and material sciences; see \cite[Chapter 4]{FeRo23} for a brief description of these models.

The regularity theory for obstacle problems \eqref{eq:OP} for integro-differential operators \eqref{eq:L} was first developed in \cite{Sil07,ACS08,CSS08} for the fractional Laplacian.
The main results of Caffarelli-Salsa-Silvestre may be summarized as follows: When $L = (-\Delta)^s$, it holds:
\begin{itemize}
\item[(1)] Solutions $u$ to \eqref{eq:OP} are $C^{1,s}$. This regularity is optimal.
\item[(2)] The free boundary $\partial \{ u > \phi\}$ splits into regular and degenerate points.
\begin{itemize}
\item[(2a)] Regular points $x_0$ are those for which the following holds:
\[\begin{split}
\sup_{B_r(x_0)} (u - \phi) \asymp r^{1+s},
\end{split}\]
and in addition 
\[\qquad \qquad \qquad \frac{(u-\phi)(x_0+rx)}{\|u-\phi\|_{L^\infty(B_r(x_0))}} \xrightarrow{\ r\to0\ } (x\cdot e)_+^{1+s} \quad \textrm{for some}\quad e\in \mathbb S^{n-1}.\]
Moreover, the set of regular points is an open subset of the free boundary, and near those points the free boundary is $C^{1,\gamma}$.
\item[(2b)] Degenerate points $x_0$ are those for which the following holds:
\[\begin{split}
0 \le \sup_{B_r(x_0)} (u - \phi) \lesssim r^{1+s+\alpha}  ~~  \text{ for some } \alpha > 0.
\end{split}\]
\end{itemize}
\end{itemize}

The case of the fractional Laplacian $(-\Delta)^s$ is quite special, because one can use the extension problem for this operator, making the obstacle problem for the fractional Laplacian equivalent to a (weighted) thin obstacle problem in $\R^{n+1}_+$. 
The identification with a local problem gives access to various tools, such as monotonicity formulas, which are known to be very useful in the proof of regularity results, such as (1) and (2). 
We refer to \cite{GaPe09}, \cite{PePe15}, \cite{KPS15}, \cite{JhNe17}, \cite{GPPS17}, \cite{FoSp18}, \cite{FeRo18}, \cite{CSV20}, \cite{CFR21}, \cite{FeJh21}, \cite{Kuk21}, \cite{SaYu21}, \cite{FeRo21}, and \cite{FeTo22} for further results on the obstacle problem for the fractional Laplacian and variants thereof, including higher regularity of free boundaries, and fine structure results for degenerate points.

The analysis of \eqref{eq:OP} becomes significantly more delicate in case $L$ is anisotropic, in the sense that it can neither be reduced to the fractional Laplacian nor be related to an equivalent local problem. 
This requires the application of new tools in order to study the regularity theory for \eqref{eq:OP}.
Recently, in \cite{CRS17} and \cite{FRS23}, new techniques have been developed to prove (1) and (2) for solutions to obstacle problems \eqref{eq:OP} governed by  nonlocal operators $L$ with kernels $K$ that are pointwise comparable to the one of the fractional Laplacian, i.e.,
\begin{align}
\label{eq:Kcomp}\tag{$K_{\asymp}$}
0<\lambda \le K(\theta) \le \Lambda \qquad \forall \theta \in \mathbb{S}^{n-1}.
\end{align}
We refer to \cite{CDS17}, \cite{AbRo20}, \cite{RoTo21}, and \cite{RTW23}, for further results on obstacle problems \eqref{eq:OP} for nonlocal operators \eqref{eq:L} satisfying \eqref{eq:Kcomp}.

Let us point out that, despite the significant recent advances in the theory, so far \textit{nothing is known} about the regularity of solutions or free boundaries for the nonlocal obstacle problem \textit{if \eqref{eq:Kcomp} is violated}. 
Some important examples of nonlocal operators of the form \eqref{eq:L}-\eqref{eq:hom} whose kernels are not pointwise comparable to the one of the fractional Laplacian are:
\[\begin{array}{l}
\displaystyle L_1 u(x) = \text{p.v.} \int_{\mathcal{C}} \big(u(x) - u(x+y)\big) \frac{\d y}{|y|^{n+2s}}, \qquad \mathcal{C}  \text{ double-cone with vertex at } 0, \vspace{4mm}\\
\displaystyle L_2u(x) = \text{p.v.} \int_{\R^n} \big(u(x) - u(x+y)\big) \frac{a(y/|y|)}{|y|^{n+2s}}\,\d y, \qquad \textrm{with}\quad a\in L^p(\mathbb S^{n-1})\setminus L^\infty(\mathbb S^{n-1}), \vspace{4mm}\\
\displaystyle L_3 u(x) = (-\partial_{x_1x_1}^2)^s u(x)  + \dots + (-\partial_{x_nx_n}^2)^s u(x).
\end{array}\]
Note that the corresponding jumping kernels
\[\begin{split}
K_1(\theta) := \mathbbm{1}_{\mathcal{C}}(\theta), \qquad K_2(\theta) := a(\theta),
\qquad K_3 := \delta_{\pm e_1} +  \delta_{\pm e_2} + \dots +  \delta_{\pm e_n}
\end{split}\]
clearly violate \eqref{eq:Kcomp} since they are either not fully supported, or possess singularities on $\mathbb{S}^{n-1}$. 
Moreover, in the absence of \eqref{eq:Kcomp}, jumping kernels can exhibit oscillatory behavior leading to the failure of the Harnack inequality; see \cite{BoSz05}, \cite{BoSz07}.

The purpose of this article is precisely to investigate the nonlocal obstacle problem \eqref{eq:OP} for operators that violate \eqref{eq:Kcomp}, and to establish for the first time optimal $C^{1,s}$-regularity estimates for solutions (1), as well as the regularity of the free boundary near regular points (2) in such general setting.

More precisely, we will consider general stable operators of the type \eqref{eq:L}-\eqref{eq:hom}; see \cite{RoSe16,FeRo23}.
The ellipticity conditions then become
\begin{align}
\label{eq:Glower}\tag{$\mathcal{G}_{\ge}$}
\inf_{e \in \mathbb{S}^{n-1}} \int_{\mathbb{S}^{n-1}} |e \cdot \theta|^2 K(\theta) \d \theta &\ge \lambda>0,\\
\label{eq:Gupper}\tag{$\mathcal{G}_{\le}$}
\Vert K \Vert_{L^1(\mathbb{S}^{n-1})} &\le \Lambda.
\end{align}
These conditions are satisfied for any $K|_{\mathbb S^{n-1}}\in L^1(\mathbb S^{n-1}) \setminus \{ 0 \}$.
In fact, in some of our results we could even allow $K$ to be a purely singular measure as in $L_3$ above; however for simplicity of notation we assume that $K$ is absolutely continuous.

Note that \eqref{eq:Glower} and \eqref{eq:Gupper} are natural conditions in the sense that they are equivalent to the comparability of the Fourier symbols of $L$ and $(-\Delta)^s$; see Lemma \ref{lemma:Fourier} below.


In some of our results, we need a slightly stronger assumption on $K$, namely
\begin{align}
\label{eq:Kupper-p}\tag{$K_{\le}^p$}
\Vert K(\cdot) \Vert_{L^p(\mathbb{S}^{n-1})} \le \Lambda
\end{align}
with $p>1$.

\subsection{Main results}

Our first main result is the following quantitative estimate which states that closeness of the solution to a blow-up of the form $(x \cdot e)_+^{1+s}$ implies local smoothness of the free boundary and local $C^{1,s}$-estimates of the solution. 
This result holds for \emph{any} stable operator $L$:

\begin{theorem}[Flatness implies $C^{1,\gamma}$]
\label{prop:4.4.15}
Let $s\in(0,1)$ and $L$ be a general stable operator of the form \eqref{eq:L}-\eqref{eq:hom}-\eqref{eq:Glower}-\eqref{eq:Gupper}. 
Let $\alpha \in (0,\min \{ s , 1-s\} )$ and $\kappa_0 > 0$. 
Then, there are $\eps > 0$, $\delta > 0$, depending only on $n,s,\lambda,\Lambda,\alpha,\kappa_0$, such that the following holds true:

Let $u \in C^{0,1}_{\rm loc}(\R^n)$ be such that
\begin{itemize}
\item[(i)] $\min \{ Lu - f , u \} = 0$ in $B_{1}$ in the distributional sense, with $|\nabla f| \le 1$,
\item[(ii)] $D^2 u \ge - \mathrm{Id}$ in $B_{1}$ with $0 \in \partial\{ u > 0 \}$,
\item[(iii)] $\Vert  \nabla u \Vert_{L^{\infty}(B_R)} \le R^{s+\alpha}$ for all $R\geq1$,
\item[(iv)] $\Vert u - \kappa (x \cdot e)_+^{1+s} \Vert_{C^{0,1}(B_{1})} \le \eps$ for some $\kappa \ge \kappa_0$ and $e \in \mathbb{S}^{n-1}$.
\end{itemize}
Then, the free boundary $\partial \{ u > 0 \}$ is a $C^{1,\gamma}$-graph in $B_{\delta}$, 
 and moreover $u\in C^{1+s}(B_\delta)$ with
\[\begin{split}
\Vert \nabla u \Vert_{C^s(B_{\delta})} \le C.
\end{split}\]
The constants 
$C$ and $\gamma > 0$ depend only on $n,s,\lambda,\Lambda,\alpha,\kappa_0$.
\end{theorem}

This type of ``flatness implies $C^{1,\gamma}$'' results are one of the crucial ingredients in the regularity theory for free boundary problems.

In case of obstacle problems, the second key ingredient is a classification of blow-ups at non-degenerate points. 
In \cite{CRS17}, the classification of blow-ups was established for nonlocal operators whose kernels satisfy \eqref{eq:Kcomp}. 
One of the main ingredients in their proof is a boundary Harnack principle. 
Unfortunately, even the interior Harnack inequality fails for general stable operators not satisfying \eqref{eq:Kcomp}, and therefore new ideas are required in order to classify blow-ups in our context.

Here, we extend for the first time the results of \cite{CRS17} to operators with kernels not satisfying \eqref{eq:Kcomp}, and more precisely, we assume $K \in L^p(\mathbb{S}^{n-1})$ for some $p > \frac{n}{2s}$.
Notice that this is completely new even for the case $p=\infty$, since we do not assume any uniformly positive lower bound as in \eqref{eq:Kcomp}.

\begin{theorem}[Classification of blow-ups]
\label{thm:classification}
Let $s\in(0,1)$ and $L$ be any stable operator of the form \eqref{eq:L}-\eqref{eq:hom} satisfying $K \neq 0$ and 
\begin{equation*}
K \in L^p(\mathbb{S}^{n-1})\qquad \textrm{for some} \quad p > \frac{n}{2s}.
\end{equation*}
Let $\alpha \in (0, \min \{ s , 1-s \})$, and $u_0 \in C^{0,1}_{loc}(\R^n)$ be such that:
\begin{itemize}
\item $u_0 \ge 0$ and $D^2 u_0 \ge 0$ in $\R^n$ with $0 \in \partial \{ u_0 > 0 \}$.
\item $u_0$
 solves in the distributional sense
\[\begin{split}
L(\nabla u_0) = 0  ~~ \text{ and } ~~ L(D_{h} u_0) \ge 0 ~~ \text{ in } \{ u_0 > 0 \} ~~ \forall h \in \R^n,
\end{split}\]
where $D_h u(x) = \frac{u(x+h) - u(x)}{|h|}$.
\item $u_0$ has controlled growth at infinity, i.e.,
\[\begin{split}
\Vert \nabla u_0 \Vert_{L^{\infty}(B_R)} \le R^{s+\alpha} \quad \textrm{for all}\quad R \ge 1.
\end{split}\]
\end{itemize}

Then, 
\[\begin{split}
u_0 = \kappa (x \cdot e)_+^{1+s}
\end{split}\]
for some $\kappa \ge 0$ and $e \in \mathbb S^{n-1}$.
\end{theorem}

By combining the previous two results, we establish the optimal regularity of solutions (1) and the regularity of free boundaries near regular points (2) for solutions to nonlocal obstacle problems \eqref{eq:OP} with kernels satisfying $K\in L^p(\mathbb S^{n-1})$ for $p > \frac{n}{2s}$. 
As said before, this is the first regularity result for the nonlocal obstacle problem with kernels that are not pointwise comparable to the one of the fractional Laplacian, and it is new even for $p=\infty$.

\begin{theorem}
\label{thm:opt-reg}
Let $s\in(0,1)$ and let $L$ be any operator of the form \eqref{eq:L}-\eqref{eq:hom} satisfying 
\eqref{eq:Glower} and \eqref{eq:Kupper-p} for some $p > \frac{n}{2s}$, and let $\alpha \in (0, \min\{ s , 1-s\} )$.

Let $\phi \in C_c^{2,\eps}(\R^n)$ with $\eps > \max\{2s-1 , 0 \}$ and let $u$ be any weak solution to the obstacle problem
\[\begin{split}
\min \{ L u, u - \phi \} = 0 ~~ \text{ in } \R^n.
\end{split}\]
Denote $C_0:=\Vert \phi \Vert_{C^{2,\eps}(\R^n)}$.
Then, we have:
\begin{itemize}
\item[(i)] $u \in C^{1+s}(\R^n)$, with 
\[\begin{split}
\Vert u \Vert_{C^{1+s}(\R^n)} \le CC_0,
\end{split}\]
where $C > 0$ depends only on $n,s,\lambda,\Lambda$.

\item[(ii)] For any free boundary point $x_0 \in \{ u > \phi\}$ there exist $c_{x_0} \ge 0$, $e \in \mathbb{S}^{n-1}$ such that
\[\begin{split}
\left| u(x) - \phi(x) - c_{x_0} \big( (x - x_0) \cdot e \big)_+^{1+s} \right| \le C C_0 |x-x_0|^{1+s+\alpha} ~~ \forall x \in B_1(x_0),
\end{split}\]
where $C > 0$ depends only on $n,s,\lambda,\Lambda,\alpha$.

\item[(iii)] Moreover, if $c_{x_0} > 0$, then the free boundary $\partial \{ u > \phi \}$ is a $C^{1,\gamma}$-graph in a neighborhood of~$x_0$, where $\gamma > 0$ depends only on $n,s,\lambda,\Lambda,\alpha$.

\end{itemize}
\end{theorem}

The optimal regularity of solutions and the study of the free boundary remain open in case $K$ merely satisfies \eqref{eq:Glower} and \eqref{eq:Gupper}, but not $K\in L^p(\mathbb S^{n-1})$ with $p > \frac{n}{2s}$.
However, thanks to \autoref{prop:4.4.15}, the only missing point in the regularity theory for general stable operators is the classification of blow-ups (i.e., to prove \autoref{thm:classification} for $p=1$). We believe that an entirely new approach is required in order to tackle this problem (see \autoref{subsec:difficulties} for a more detailed discussion).

\subsection{Difficulties and strategy of proof}
\label{subsec:difficulties}

The analysis of obstacle problems for nonlocal operators with kernels \emph{not} satisfying \eqref{eq:Kcomp}  comes with two main difficulties:

\subsubsection*{Lack of full support} First, although $L$ is non-degenerate due to \eqref{eq:Glower}, the  kernel $K$ might not be fully supported. 
As a consequence, it is not possible to establish locally uniform lower bounds on $L$ applied to bump functions, since the bump might not be seen by the operator from all points. 
Such estimates are used several times in \cite{CRS17}, in particular in order to perform barrier arguments.
Here, we circumvent this issue by suitably adjusting the shape of the bump functions, depending on the geometry of the problem (see for instance the proof of \autoref{thm:bdryHarnack-cone}), together with a quantification of the directions in which the kernel is possibly degenerate (see \autoref{lemma:K-mass}).

\subsubsection*{Failure of the Harnack inequality} Second, and more importantly, a key tool in \cite{CRS17} is the Harnack inequality, and more precisely the two ``half Harnacks'':
\[\left.\begin{array}{r} Lu \geq0 \quad\textrm{in}\quad B_2 \\
u\geq0 \quad\textrm{in}\quad \R^n \end{array} \right\}
\qquad \Longrightarrow \qquad 
\inf_{B_{1}} u \geq c\int_{\R^n} \frac{u(x)}{1+|x|^{n+2s}} \d x,\]
and 
\[Lu \leq0 \quad\textrm{in}\quad B_2 \qquad \Longrightarrow \qquad 
\sup_{B_{1}} u \leq C\int_{\R^n} \frac{|u(x)|}{1+|x|^{n+2s}} \d x.\]
Both of them fail, in general, for kernels not satisfying  \eqref{eq:Kcomp}.

In a sense, the lack of two-sided pointwise bounds as in \eqref{eq:Kcomp} makes the possible anisotropy of the kernel more severe. 
Namely, the operators we consider might still exhibit anisotropy after averaging out the kernels over points close by (oscillating long jumps). This phenomenon leads to a failure of the Harnack inequality
\[\left.
\begin{array}{r} Lu = 0 \quad\textrm{in}\quad B_2 \\
u\geq0 \quad\textrm{in}\quad \R^n \end{array} \right\}
\qquad \Longrightarrow \qquad  \sup_{B_1} u \le C \inf_{B_1} u .
\]
The failure of the Harnack inequality implies that also the boundary Harnack principle ceases to hold.

The only result in this direction that holds true for general stable operators is a \textit{weak Harnack} inequality.
This is a key ingredient in our proofs, as explained below.

\subsubsection*{Strategy of the proof}

Even though the boundary Harnack inequality fails for general stable operators, here we establish a particular version of the boundary Harnack principle, which holds for monotone solutions outside convex cones (see \autoref{thm:bdryHarnack-cone}). 
This turns out to be sufficient in order to classify blow-ups and holds true for kernels satisfying 
\eqref{eq:Kupper-p} with $p > \frac{n}{2s}$.
Its proof makes heavy use of the following {weak Harnack} inequality (see \autoref{lemma:wHI-p})
\[\begin{split}
\left(\dashint_{B_1} u^q \right)^{1/q} \le C \inf_{B_1} u, \qquad \text{ where } q < \frac{n}{n-2s},
\end{split}\]
and local boundedness estimate (see \autoref{lemma:locbd-p})
\[\begin{split}
\sup_{B_{1}} u \le C \left(\dashint_{B_2} u \right) + C \sup_{R \ge 2} \frac{\left( \dashint_{B_{R}} |u|^{p'} \right)^{1/p'}}{R^{2s-\eps}},
\end{split}\]
which remain true for $L$-harmonic functions that are globally nonnegative, if \eqref{eq:Glower} and \eqref{eq:Kupper-p} are satisfied. 
To deduce our new boundary Harnack type principle, we establish a growth control on the solution at infinity with the help of the weak Harnack inequality and a barrier argument (see \autoref{lemma:L1-growth-control}). This allows us to estimate the second term in the local boundedness estimate. The restriction $p > \frac{n}{2s}$ comes from the (sharp) condition on $q = p'$ in the weak Harnack inequality. 

On the other hand, in order to obtain the $C^{1,\gamma}$-regularity of the free boundary near regular points, we establish that the free boundary is Lipschitz with a Lipschitz constant depending on the closeness of the solution to the blow-up (see the assumption of the quantitative estimate \autoref{prop:4.4.15}). Moreover, this way, the Lipschitz constant can be made arbitrarily small and it turns out that the free boundary is flat Lipschitz. This information is sufficient for blow-up arguments to work, which yield pointwise boundary regularity estimates as in \autoref{lemma:proposition5.4} (see \cite[Proposition 5.4]{RoSe17}). This way, one can obtain $C^{1,\gamma}$-regularity of the free boundary for general kernels, avoiding a boundary Harnack principle.

\subsubsection*{Solution concepts} 

Finally, let us point out that throughout the paper we will work with weak or distributional solutions, but will never use viscosity solutions.
The reason for this is that most known results for viscosity solutions are developed for kernels satisfying \eqref{eq:Kcomp} (see, e.g., \cite{FeRo23}), and thus do not apply in our setting.

This choice calls for some technical results such as Lemmas \ref{lemma:truncation-subsol}, \ref{lemma:distr-ptw-evaluate}, \ref{lemma:ptw-distr}, and \ref{lemma:least-super}, which are elementary in nature, but might still be of interest to some readers.

\subsection{Acknowledgments}

The authors were supported by the European Research Council (ERC) under the Grant Agreement No 801867, and by the AEI project PID2021-125021NA-I00 (Spain).
In addition, the first author was supported by the AGAUR project 2021 SGR 00087 (Catalunya), the AEI grant RED2022-134784-T funded by MCIN/AEI/10.13039/501100011033 (Spain), and the AEI Mar\'ia de Maeztu Program for Centers and Units of Excellence in R\&D (CEX2020-001084-M).

\subsection{Outline}

This article is structured as follows: In Section \ref{sec:prelim}, we prove several auxiliary results on general elliptic stable operators and introduce solution concepts. In Section \ref{sec:basic}, we establish semiconvexity and $C^{1,\tau}$-regularity for solutions to the obstacle problem, proving \autoref{thm:C1-tau}. The proof of the classification of blow-ups (see \autoref{thm:classification}) is given in Section \ref{sec:classification}. Section \ref{sec:regularity} is dedicated to the proof of the quantitative estimate (see \autoref{prop:4.4.15}) from which we deduce our main result, \autoref{thm:opt-reg}.

\section{Preliminaries}
\label{sec:prelim}

The goal of this section is to establish several auxiliary results that will be used in the course of this article. We start by discussing in more detail the kernels $K$ considered in this article and establish some helpful properties that follow from the assumptions \eqref{eq:Kupper-p} and \eqref{eq:Glower}. Second, we introduce appropriate weak and distributional solution concepts.

The following function space captures some information on the growth of functions at infinity. For $\alpha \in (0,s)$ and $1 \le q \le \infty$, we introduce
\[\begin{split}
L^{q}_{s+\alpha}(\R^n) = \left\{ u \in L^q_{loc}(\R^n) : \Vert u \Vert_{L^{q}_{s+\alpha}(\R^n)} := \sup_{R \ge 1} \frac{\left( \dashint_{B_R} |u|^q \d x \right)^{1/q}}{R^{s+\alpha}} < \infty \right\}.
\end{split}\]

\subsection{Properties of kernels}

In this section, we collect several preliminary results on kernels $K$ satisfying the assumptions \eqref{eq:Kupper-p}, resp. \eqref{eq:Gupper}, and \eqref{eq:Glower}. We denote $q = p' = \frac{p}{p-1}$.


\begin{lemma}[see Proposition 2.2.1 in \cite{FeRo23}]
\label{lemma:Fourier}
The following are equivalent:
\begin{itemize}
\item[(i)] \eqref{eq:Glower} and \eqref{eq:Gupper} hold true.
\item[(ii)] There exist $0 < c_1 \le c_2$ such that
\[\begin{split}
c_1 |\xi|^{2s} \le \mathcal{A}_K(\xi) \le c_2 |\xi|^{2s},
\end{split}\]
where $\mathcal{A}_K$ denotes the Fourier symbol of $K$, given by
\[\begin{split}
\mathcal{A}_K(\xi) = \frac{1}{2}\int_{\R^n} \big(1- \cos(y \cdot \xi)\big) K(y) \d y.
\end{split}\]
\end{itemize}
\end{lemma}

Let us remark that the condition \eqref{eq:Kupper-p} can equivalently be rewritten as follows:
\begin{align}
\label{eq:Kupper-p-equiv}
R^{n\left(1 - \frac{1}{p} \right)} \left(\int_{\R^n \setminus B_R} K^{p}(y) \d y \right)^{1/p} \le \Lambda R^{-2s}.
\end{align}

Next, we introduce the tail, with respect to $K$. This object stores the information on a function $u$ at infinity with respect to the ball $B_r(x_0)$ for some $r > 0$ and $x_0 \in \R^n$:
\[\begin{split}
\tail_K(u;r,x_0) := r^{2s} \int_{\R^n \setminus B_r(x_0)} |u(y)| K(y) \d y.
\end{split}\]
In case $x_0 = 0$, we will simply write $\tail_K(u;r,x_0) = \tail_K(u;r)$. We have the following estimate:

\begin{lemma}
\label{lemma:tail-est-p}
Assume \eqref{eq:Kupper-p} for some $1 \le p \le \infty$. Then, for any $0 < \eps < 2s$ and $u : \R^n \to \R$
\begin{align}
\label{eq:tail-est-p}
\tail_K(u;r) := r^{2s} \int_{\R^n \setminus B_r} |u(y)| K(y) \d y \le c r^{2s-\eps} \sup_{R \ge r} \frac{\left( \dashint_{B_R} |u|^{q} \right)^{1/q}}{R^{2s-\eps}},
\end{align}
where $c = c(n,s,\eps,\Lambda) > 0$. Thus, $\tail_K(u;r) < \infty$ whenever $u \in L^q_{s+\alpha}(\R^n)$ for some $\alpha \in (0,s)$.
\end{lemma}

\begin{proof}
We compute using \eqref{eq:Kupper-p-equiv}
\[\begin{split}
\tail_K(u;r) &= r^{2s} \sum_{k = 0}^{\infty} \int_{B_{2^k r} \setminus B_{2^{k-1}r}} |u(y)| K(y) \d y\\
&\le c r^{2s} \sum_{k = 0}^{\infty} \left( \dashint_{B_{2^k}r} |u(y)|^q \d y \right)^{1/q} (2^k r)^{n \left( 1 - \frac{1}{p} \right)} \left( \int_{\R^n \setminus B_{2^k r}} K^p(y) \d y \right)^{1/p}\\
&\le c r^{2s} \sum_{k = 0}^{\infty} \left( \dashint_{B_{2^k}r} |u(y)|^q \d y \right)^{1/q} (2^k r)^{-2s}\\
&\le c r^{2s} \sum_{k = 0}^{\infty} (2^k r)^{-\eps} \sup_{k} \left[ \frac{\left(\dashint_{B_{2^k}r} |u(y)|^q \d y\right)^{1/q}}{(2^k r)^{2s-\eps}} \right]\\
&\le c r^{2s-\eps} \sup_{R \ge r} \frac{\left( \dashint_{B_R} |u|^q \d x \right)^{1/q}}{R^{2s-\eps}}.
\end{split}\]
\end{proof}

The following lemma, distills a useful property out of \eqref{eq:Glower} and \eqref{eq:Gupper}. In fact, it allows us to locate the mass of $K$ on the sphere, thereby giving us some important information if we want to give a pointwise bound on $Lu(x_0)$ for some $x_0 \in \R^n$.

\begin{lemma}
\label{lemma:K-mass}
Assume \eqref{eq:Glower} and \eqref{eq:Gupper}. Let $K$ be homogeneous. Then, there exists $\delta_0 > 0$ depending only on $\lambda,\Lambda$ such that for any $e \in \mathbb{S}^{n-1}$:
\[\begin{split}
\int_{\{e \cdot \theta \ge \delta_0 \}} K(\theta) \d \theta \ge \lambda/2.
\end{split}\]
In particular, for any $r > 0$, we obtain
\[\begin{split}
\int_{\left \{ \frac{y}{|y|} \cdot e \ge \delta_0 \right\} \cap (B_{2r} \setminus B_r)} K(y) \d y \ge c r^{-2s}
\end{split}\]
for some $c > 0$ depending only on $n,s,\lambda,\Lambda$.
\end{lemma}

\begin{proof}
We compute
\[\begin{split}
\int_{\{|e \cdot \theta| \ge \delta_0 \}} K(\theta) \d \theta &\ge \int_{\{|e \cdot \theta| \ge \delta_0 \}} |e \cdot \theta|^2 K(\theta) \d \theta \\
&= \int_{\mathbb{S}^{n-1}} |e \cdot \theta|^2 K(\theta) \d \theta - \int_{\{|e \cdot \theta| \le \delta_0 \}} |e \cdot \theta|^2 K(\theta) \d \theta \\
&\ge \lambda - \delta_0^2 \Lambda
\end{split}\]
and deduce the desired result upon choosing $\delta_0 < \sqrt{\lambda /(2\Lambda)}$ and using that $K$ is symmetric.
\end{proof}

\subsection{Solution concepts}

Throughout this article, we will deal with distributional and weak solutions, which we introduce in the sequel.
 
First, we define the bilinear associated with $L$ as follows
\[\begin{split}
\cE^K(u,u) := \int_{\R^n} \int_{\R^n} (u(x) - u(y))^2 K(x-y) \d y \d x,
\end{split}\]
and observe that $\cE^K(u,u) < \infty$, whenever $u \in H^s(\R^n)$,  due to \autoref{lemma:Fourier}.

\begin{definition}
\label{def:weak-sol}
Let $\Omega \subset \R^n$ be an open domain. Let $f \in L^{\infty}(\Omega)$. We say that $u$ is a weak subsolution to $Lu \le f$ in $\Omega$ if $u \in L^2(\Omega)$ and
\begin{align}
\label{eq:weak-form}
\cE^K(u,\phi) \le \int_{\R^n} f \phi \d x ~~ \forall \phi \in H^s(\R^n) ~~ \text{ with compact } \supp(\phi) \subset \Omega, ~~ \text{ and } \phi \ge 0,
\end{align}
and it holds
\begin{align}
\label{eq:energy-finite}
\cE^K_{\R^n,\Omega}(u,u) := \iint_{(\R^n \times \R^n) \setminus (\Omega \times \Omega)} (u(x) - u(y))^2 K(x-y) \d y \d x < \infty.
\end{align}
We say that $u$ is a weak supersolution to $Lu \ge f$ in $\Omega$ if $u$ satisfies \eqref{eq:energy-finite} and \eqref{eq:weak-form} holds true for any $\phi \in H^s(\R^n)$ with compact $\supp(\phi) \subset \Omega$ and $\phi \le 0$. We say that $u$ is a weak solution to $Lu = f$ in $\Omega$ if $u$ is a weak subsolution and a weak supersolution.
\end{definition}

The following lemma recalls a basic fact about weak solutions. We provide a short proof since we were not able to find a reference in the literature for general kernels satisfying only \eqref{eq:Glower} and \eqref{eq:Gupper}.

\begin{lemma}
\label{lemma:truncation-subsol}
Assume \eqref{eq:Glower} and \eqref{eq:Gupper}. Let $\Omega \subset \R^n$ be an open, bounded domain. Assume that $u \in L^2(\Omega)$ is a weak solution to $L u = 0$ in $\Omega$. Let $\delta \ge 0$. Then $u_{\delta} := \max \{ u,\delta \}$ satisfies $L u_{\delta} \le 0$ in $\Omega$ in the weak sense.
\end{lemma}

\begin{proof}
The proof is standard. First, we observe that it suffices to consider the case $\delta = 0$ since $\cE^K(u-\delta,\phi) = \cE^K(u,\phi)$. The idea is to approximate $F(x) = x_+$ by smooth, convex, non-decreasing functions $F_k :\R \to [0,\infty)$ satisfying $F_k(x) = F_k'(x) = 0$ for $x \le 0$, $F_k \to F$ uniformly, and $\sup_k \Vert F_k' \Vert_{\infty} \le C < \infty$. Then, upon the observation (see for instance \cite[Lemma 2.3]{Kas09}) that by convexity of $F_k$ for any $a,b \in \R$ and $\phi_1,\phi_2 \ge 0$ it holds
\[\begin{split}
(F_k(a) - F_k(b))(\phi_1 - \phi_2) \le (a-b)(\phi_1 F_k'(a) - \phi_2 F_k'(b)),
\end{split}\]
we obtain for any $\phi \in H^s(\R^n)$ with $\phi \equiv 0$ in $\R^n \setminus \Omega$ and $\phi \ge 0$:
\[\begin{split}
\cE^K(F_k(u),\phi) \le \cE^K(u,F_k'(\phi)) \le 0.
\end{split}\]
The second inequality follows from the fact that $u$ satisfies $Lu \le 0$ in $\Omega$ in the weak sense, using $F_k'(\phi) \in H^s(\R^n)$ as a test function. Finally, we observe that 
\[\begin{split}
F_k(u) \to u_+ ~~ \text{ in } L^2(\Omega)
\end{split}\]
by dominated convergence, using $|F_k(u)| \le C u_+$. Moreover, $C^{-1} F_k(u)$ is a normal contraction of $u_+$, which yields that
\[\begin{split}
\sup_{k} \iint_{(\R^n \times \R^n) \setminus (\Omega \times \Omega)} &(F_k(u)(x) - F_k(u)(y))^2 K(x-y) \d y \d x\\
& \le C^2 \iint_{(\R^n \times \R^n) \setminus (\Omega \times \Omega)} (u_+(x) - u_+(y))^2 K(x-y) \d y \d x < \infty.
\end{split}\]
Thus, by weak compactness of the separable Hilbert space
\[\begin{split}
\left(\{u \in L^2(\Omega) : \cE^K_{\R^n,\Omega}(u,u) < \infty \}, ~~ \Vert \cdot \Vert_{L^2(\Omega)} + (\cE^K_{\R^n,\Omega}(\cdot,\cdot))^{1/2}\right),
\end{split}\]
we obtain that for any $\phi \in H^s(\R^n)$ with $\supp(\phi) \subset  \Omega$ it holds (up to a subsequence)
\[\begin{split}
0 \ge \cE^K(F_k(u),\phi) \to \cE^K(u_+,\phi), ~~ \text{ as } k \to \infty,
\end{split}\]
which implies the desired result.
\end{proof}

Next, we introduce the notion of distributional solutions:

\begin{definition}
\label{def:distr-sol}
Let $\Omega \subset \R^n$ be an open domain and $\eps \in (0,2s)$. Let $f \in L^{1}_{loc}(\Omega)$. We say that $u \in L^{\infty}_{2s-\eps}(\R^n)$ is a distributional subsolution to $Lu \le f$ in $\Omega$ if
\begin{align}
\label{eq:distr-form}
\int_{\R^n} (L \eta) u \le \int_{\R^n} \eta f ~~ \forall \eta \in C_c^{\infty}(\R^n) ~~ \text{ with compact } \supp(\eta) \subset \Omega, ~~ \text{ and } \eta \ge 0.
\end{align}
We say that $u$ is a distributional supersolution to $Lu \ge f$ in $\Omega$ if $u$ satisfies \eqref{eq:distr-form} $\eta \in C_c^{\infty}(\R^n)$ with compact $\supp(\eta) \subset \Omega$ and $\eta \le 0$. We say that $u$ is a weak solution to $Lu = f$ in $\Omega$ if $u$ is a distributional subsolution and supersolution.
\end{definition}

We prove the following lemma, which says that distributional supersolutions can be treated just as classical supersolutions at a point $x_0$, if the supersolution can be touched from below by a $C^2$ function. Note that such property is trivial for viscosity solutions.

\begin{lemma}
\label{lemma:distr-ptw-evaluate}
Assume \eqref{eq:Glower} and \eqref{eq:Gupper}. Let $C > 0$ and $\eps \in (0,2s)$, and assume that $u \in L^{\infty}_{2s-\eps}(\R^n)$ is locally uniformly H\"older continuous, i.e. $\sup_{x_0 \in \R^n}[u]_{C^{\tau}(B_1(x_0))} < \infty$ for some $\tau \in (0,1)$, and satisfies in the distributional sense for some $C > 0$
\[\begin{split}
L u \ge -C ~~ \text{ in } B_1.
\end{split}\]
Moreover, assume that there exists a function $\phi \in L^{\infty}_{2s-\eps}(\R^n)$ that is $C^2$ around $0$, and such that $u(0) = \phi(0)$, and $u \ge \phi$ in $\R^n$.
Then, $L u(0) \ge -C$ in the classical sense.
\end{lemma}

\begin{proof}
First, we prove that for any $\Phi \in L^{\infty}_{2s-\eps}(\R^n)$ that is $C^2$ around $0$ and also satisfies \\$\sup_{x_0 \in \R^n}[\Phi]_{C^{\tau}(B_1(x_0))} < \infty$, and $\Phi(0) = u(0)$, $\Phi \le u$, it holds
\begin{align}
\label{eq:distr-eval-help1}
L \Phi(0) \ge - C.
\end{align}
By contradiction, assume that there exist such $\Phi$ and $r > 0$ such that $L \Phi \le -C - r$ in $B_r$. Note that for the last property, we used that $L \Phi$ is H\"older continuous in $B_{1/2}$ due to the regularity of $\Phi$ and \cite[Lemma 2.2.5 (ii), Remark 2.2.6]{FeRo23}.
Let $\delta > 0$ and observe that for $w_{\delta} = u - \Phi - \delta$ and $\delta > 0$ it holds:
\[\begin{split}
L w_{\delta} \ge r ~~ \text{ in } B_r ~~ \text{ in the distributional sense}.
\end{split}\]
Let now $\psi \in C_c^{\infty}(B_r)$ be nonnegative and define $v \in H^s(\R^n)$ to be the unique weak solution (see \cite[Theorem 2.2.19]{FeRo23}, \cite{FKV15}) to
\[\begin{split}
\begin{cases}
L v &= \psi ~~ \text{ in } B_r,\\ 
v &= 0 ~~ \text{ in } \R^n \setminus B_r.
\end{cases}
\end{split}\]
Note that $v \in C^{2s+\alpha}_{loc}(B_r)$ for some $\alpha \in (0,1)$, and hence is a strong solution. Indeed, $v \in C^{\alpha}(\R^n)$ by boundary regularity theory (see \cite[Proposition 2.5.10]{FeRo23}), and thus, by interior regularity theory, we have $v \in C^{2s+\alpha}_{loc}(B_r)$ (see \cite[Proposition 2.4.4]{FeRo23}). Moreover, by the maximum principle (see \cite[Lemma 2.3.3]{FeRo23}), it holds $v \ge 0$ in $\R^n$. Therefore,
\[\begin{split}
(w_{\delta}, \psi) = (w_{\delta} , L v) \ge r \Vert v \Vert_{L^1(B_r)} \ge 0.
\end{split}\]
Note that $v$ is an admissible test function by classical density results. Since $\psi \ge 0$, it follows that $w_{\delta} \ge 0$ a.e. in $B_r$. This is a contradiction, since $w_{\delta}(0) = - \delta$ and $w_{\delta}$ is continuous at 0. We have shown \eqref{eq:distr-eval-help1}.

Let us now turn to the actual proof of the lemma. Let $\phi$ be as in the assumption. First of all, note that $u(0) - u(x) \le \phi(0) - \phi(x)$ for any $x \in \R^n$. Therefore, $Lu(0) \le L \phi(0) < \infty$, and $Lu(0)$ can be evaluated in a pointwise way, however note that it could be $Lu(0) = - \infty$. 
It remains to show that $L u(0) \ge -C$. To see this, let us define for $\delta > 0$ small enough
\[\begin{split}
\phi_{\delta} = 
\begin{cases}
\phi ~~ &\text{ in } B_{\delta},\\
u ~~ &\text{ in } \R^n \setminus B_{\delta}.
\end{cases}
\end{split}\]
Note that $\phi_{\delta} \nearrow u$, as $\delta \to 0$, and therefore, by monotone convergence, it holds $L (\phi_{\delta}- u)(0) \searrow 0$. Moreover, note that \eqref{eq:distr-eval-help1} also holds true for $\phi_{\delta}$ since it can be approximated by functions $\Phi \in L^{\infty}_{2s-\eps}(\R^n)$ that are $C^2$ around $0$ and satisfy $\sup_{x_0 \in \R^n}[\Phi]_{C^{\tau}(B_1(x_0))} < \infty$, and $\Phi(0) = u(0)$, $\Phi \le u$. In fact, given any $\gamma \in (0,\delta)$, we can find such $\Phi$ satisfying $|L \Phi(0) - L \phi(0)| < \gamma$ by choosing $\Phi = \phi$ in $B_{\delta - \gamma} \cup (\R^n \setminus B_{\delta})$ and doing a $C^{\tau}$-interpolation between $u,\phi$ on $B_{\delta} \setminus B_{\delta - \gamma}$ and using \eqref{eq:Gupper}.\\
Thus, it follows from an application of \eqref{eq:distr-eval-help1} to $\phi_{\delta}$
\[\begin{split}
0 = \lim_{\delta \to 0} L (\phi_{\delta}- u)(0) \ge \liminf_{\delta \to 0} L \phi_{\delta}(0) - L u(0) \ge - C - Lu(0),
\end{split}\]
which implies the desired result.
\end{proof}

The following lemma relates pointwise supersolutions to distributional supersolutions:

\begin{lemma}
\label{lemma:ptw-distr}
Assume \eqref{eq:Glower} and \eqref{eq:Gupper}. Let $C > 0$ and $\eps \in (0,2s)$, and assume that $u \in L^{\infty}_{2s-\eps}(\R^n)$ satisfies in the pointwise sense
\[\begin{split}
L u \ge - C ~~ \text{ in } B_1.
\end{split}\]
Then, $L u \ge -C$ in $B_1$ also in the distributional sense.
\end{lemma}

Note that in the situation of the above lemma, it could be that $L u = + \infty$ for some points in $B_1$, since $u$ is not assumed to be smooth.

\begin{proof}
The proof follows from the observation that for any $\eta \in C_c^{\infty}(\R^n)$ with $\eta \ge 0$ it holds
\[\begin{split}
-C \int_{\R^n} \eta \d x \le \int_{\R^n} \eta (L u) \d x = \int_{\R^n} (L \eta) u \d x.
\end{split}\]
Here, we used the integration by parts formula, which can be shown in the same way as in \cite[Lemma 2.2.23]{FeRo23}.
\end{proof}

\subsection{Weak Harnack inequality and local boundedness}

Under a pointwise comparability condition on $K$, it is well-known that solutions to $Lu = 0$ in $\Omega$ satisfy a Harnack inequality. 
In our more general framework, we cannot expect a Harnack inequality to hold true. However, we can establish a weak Harnack inequality, and a local boundedness estimate including a nonlocal tail term for weak super-/ subsolutions. 

First, we prove a weak Harnack inequality for nonnegative weak supersolutions: 

\begin{lemma}[weak Harnack inequality]
\label{lemma:wHI-p}
Assume \eqref{eq:Glower}, \eqref{eq:Gupper}, and $2s < n$. Let $u$ be a globally nonnegative, weak supersolution to $Lu \ge 0$ in $B_{2}$. Then, for any $R \in (0,1)$ and $0 < q < \frac{n}{n-2s}$:
\begin{align}
\label{eq:wHI-p}\tag{wHI}
\left(\dashint_{B_R} u^q \d x \right)^{1/q} \le c \inf_{B_R} u,
\end{align}
where $c > 0$ is a constant depending only on $n,s,\lambda,\Lambda,q$, which explodes as $q \to \frac{n}{n-2s}$.
\end{lemma}

\begin{proof}
First, by \cite[Theorem 1.6]{ImSi20}, there exist $\eps \in (0,1)$ and $c > 0$, depending only on $n,s,\lambda,\Lambda$, such that
\begin{align}
\label{eq:whi-Leps}
\left(\dashint_{B_R} u^{\eps} \d x \right)^{1/\eps} \le c \inf_{B_R} u.
\end{align}
Note that \cite{ImSi20} assumes $u$ to be globally bounded instead of just \eqref{eq:energy-finite}. However, this issue can be circumvented by a standard truncation argument, using that $\tilde{u} = u \mathbbm{1}_{B_2}$ is still a weak supersolution to $L \tilde{u} \ge 0$ in $B_{3/2}$, since $- L(u \mathbbm{1}_{\R^n \setminus B_2}) \ge 0$ in $B_{3/2}$. \\
Moreover, we claim that, using a Moser iteration scheme for small positive exponents, one can prove that
\begin{align}
\label{eq:Moser-positive}
\left(\dashint_{B_{R/2}} u^q \d x \right)^{1/q} \le c \left(\dashint_{B_R} u^{\eps} \d x \right)^{1/\eps},
\end{align}
where $c > 0$ depends only on $n,s,\lambda,\Lambda,\eps$. To prove \eqref{eq:Moser-positive}, we will follow the arguments in the proof of \cite[Theorem 4.2]{KaWe22}, which are established for kernels satisfying \eqref{eq:Gupper} and a Sobolev embedding, i.e.,
\begin{align}
\label{eq:Sob}
\Vert v^2 \Vert_{L^{\frac{n}{n-2s}}(\R^n)} \le c \int_{\R^n} \int_{\R^n} (v(x) - v(y))^2 K(x-y) \d y \d x \qquad \forall v \in L^{\frac{2n}{n-2s}}(\R^n).
\end{align} 
Note that \eqref{eq:Sob} is satisfied in our setting. Indeed, by Fourier transform, we can rewrite
\[\begin{split}
\int_{\R^n} \int_{\R^n} (v(x) - v(y))^2 K(x-y) \d y \d x  = \int_{\R^n} |\mathcal{F} v(\xi)|^2 \mathcal{A}_K(\xi)\d \xi \ge c \int_{\R^n} |\mathcal{F} v(\xi)|^2 |\xi|^{2s} \d \xi,
\end{split}\]
where $\mathcal{A}_K$ denotes the Fourier symbol of $L$ and we used \autoref{lemma:Fourier}.
Therefore, \eqref{eq:Sob} follows from the classical fractional Sobolev embedding, i.e.
\[\begin{split}
\Vert v^2 \Vert_{L^{\frac{n}{n-2s}}(\R^n)} \le c [v]_{H^s(\R^n)}^2 = c \int_{\R^n} |\mathcal{F}v(\xi)|^2 |\xi|^{2s} \d \xi \le c \int_{\R^n} \int_{\R^n} (v(x) - v(y))^2 K(x-y) \d y \d x.
\end{split}\]

We are now in the position to apply the considerations in \cite{KaWe22} to our setting. By following the arguments in \cite[Proof of Theorem 4.2]{KaWe22} (and translating them to elliptic equations), we obtain 
\[\begin{split}
\Vert u \Vert_{L^{p\kappa}(B_r)} \le \left( c \rho^{-2s} \right)^{1/p} \Vert u \Vert_{L^p(B_{r+\rho})}
\end{split}\]
for any $0 < \rho \le r \le r +\rho \le R$ and $\eps < p < q/\kappa$, where $c > 0$ depends on $n,s,\lambda,\Lambda,1-\frac{q}{\kappa}$. Moreover, note that in the elliptic case we have $\kappa = \frac{n}{n-2s}$. From here, \eqref{eq:Moser-positive} follows by a standard iteration argument upon determining $K \in \N$ to be the unique number such that 
\begin{align}
\label{eq:K-choice}
\frac{q}{\kappa}\kappa^{-K} \le \eps < \frac{q}{\kappa} \kappa^{-K+1},
\end{align}
and choosing $p_i = \frac{q}{\kappa}\kappa^{-i}$, $\rho_i = 2^{-i-1}R$, $r_K = R$, and $r_{i-1} = r_{i} - \rho_i$ for any $0 \le i \le K$. In fact, we obtain
\[\begin{split}
\Vert u \Vert_{L^q(B_{R/2})} &\le \Vert u \Vert_{L^{p_0 \kappa}(B_{r_0})} \le (2^{2s}c)^{\frac{\kappa}{q}\sum_{j = 1}^K \kappa^j}  2^{\frac{2s\kappa}{q} \sum_{j = 1}^K j \kappa^j} R^{-\frac{2s\kappa}{q} \sum_{j=1}^K \kappa^j} \Vert u \Vert_{L^{\frac{q}{\kappa}\kappa^{-K}}(B_R)}\\
&\le R^{-\frac{n \kappa}{q} (\kappa^K - 1)} \Vert u \Vert_{L^{\frac{q}{\kappa}\kappa^{-K}}(B_R)},
\end{split}\]
where we used that
\[\begin{split}
-\frac{2s\kappa}{q}\sum_{j=1}^K \kappa^j = -\frac{2s\kappa}{q} \left(\frac{1 - \kappa^{K+1}}{1-\kappa} - 1 \right) = -\frac{2s\kappa}{q} \frac{\kappa}{\kappa-1}(\kappa^K - 1) = -\frac{n \kappa}{q} (\kappa^K - 1).
\end{split}\]
By \eqref{eq:K-choice}, this yields
\[\begin{split}
\left( \dashint_{B_{R/2}} u^q \d x \right)^{1/q} \le c \left( \dashint_{B_{R}} u^{\frac{q}{\kappa}\kappa^{-K}} \d x \right)^{\frac{\kappa}{q} \kappa^K} \le c \left( \dashint_{B_{R}} u^{\eps} \d x \right)^{1/\eps},
\end{split}\]
as desired. By combination of \eqref{eq:Moser-positive} with \eqref{eq:whi-Leps}, we conclude the proof.
\end{proof}


Next, we also have the following local boundedness estimate for weak subsolutions:

\begin{lemma}[local boundedness]
\label{lemma:locbd-p}
Assume \eqref{eq:Glower} and \eqref{eq:Kupper-p} for some $1 \le p \le \infty$, and $2s < n$. Let $u$ be a weak subsolution to $L u \le 0$ in $B_2$. Then, for any $R \in (0,1)$ and any $\eps \in (0,2s)$
\begin{align}
\label{eq:locbd-prelim}
\sup_{B_{R/2}} u  \le c \left(\dashint_{B_R} |u|^{2}\right)^{1/2} + c R^{2s-\eps} \sup_{\bar{R} \ge R} \frac{\left( \dashint_{B_{\bar{R}}} |u|^{q} \right)^{1/q}}{\bar{R}^{2s-\eps}},
\end{align} 
where $c > 0$ depends only on $n,s,\lambda,\Lambda,q$. Moreover, if $u \in L^{q}_{2s-\eps}(\R^n)$ and $u \ge 0$ in $B_R$, then
\begin{align}
\label{eq:locbd-p}
\sup_{B_{R/2}} u \le c \left(\dashint_{B_R} u \right) + c R^{2s-\eps} \sup_{\bar{R} \ge R} \frac{\left( \dashint_{B_{\bar{R}}} |u|^{q} \right)^{1/q}}{\bar{R}^{2s-\eps}},
\end{align}
where $c > 0$ depends only on $n,s,\lambda,\Lambda,q$.
\end{lemma}

\begin{proof}
We follow the arguments in the proofs of \cite[Theorem 3.6, resp. Theorem 6.1]{KaWe22b} after translating them to the elliptic setting. In fact, for any $0 < k < l$ and $R/2 \le r \le R$ and $0 < \rho \le r \le r+\rho \le R$, we obtain
\[\begin{split}
\Vert (u-l)_+^2 \Vert_{L^1(B_r)} \le c (l-k)^{-\frac{4s}{n}} \rho^{-2s} \left( 1 + \frac{\sup_{x \in B_{r+\frac{\rho}{2}}}\tail_K(u;\rho/2,x)}{l-k} \right) \Vert (u-k)_+^2 \Vert_{L^1(B_{r+\rho})}^{1 + \frac{2s}{n}}.
\end{split}\]
This result holds true for kernels satisfying \eqref{eq:Gupper} and the Sobolev inequality \eqref{eq:Sob}. Note that \eqref{eq:Sob} holds true in our case by the same argument, as in the proof of \autoref{lemma:wHI-p}.\\
Next, let us apply \autoref{lemma:tail-est-p} to obtain:
\[\begin{split}
\sup_{x \in B_{r+\frac{\rho}{2}}}\tail_K(u;\rho/2,x) &\le c\sup_{x \in B_{r+\frac{\rho}{2}}} \left[ \rho^{2s-\eps} \sup_{\bar{R} \ge \rho/2} \frac{\left( \dashint_{B_{\bar{R}}(x)} |u|^q \d x\right)^{1/q}}{\bar{R}^{2s-\eps}}\right]\\
&\le c \left(\frac{R}{\rho}\right)^{\frac{n}{q}} \sup_{\bar{R} \ge R} \frac{\left( \dashint_{B_{\bar{R}}} |u|^q \d x\right)^{1/q}}{\bar{R}^{2s-\eps}}.
\end{split}\]
From here, the proof follows by a standard iteration argument upon defining $l_i = M(1-2^{-i})$, $l_0 = 0$, and $\rho_i = 2^{-i-1}R$, $r_{i+1} = r_i - \rho_{i+1}$, $r_0 = R$, and $A_i = \Vert (u-l_i)_+^2 \Vert_{L^1(B_{r_i})}$, where
\[\begin{split}
M = \sup_{\bar{R} \ge R} \frac{\left( \dashint_{B_{\bar{R}}} |u|^q \d x\right)^{1/q}}{\bar{R}^{2s-\eps}} + C R^{-\frac{n}{2}} A_0^{1/2}
\end{split}\]
for a large enough constant $C > 0$, depending only on $n,s,\lambda,\Lambda,q$. In fact, the aforementioned estimates and choices yield for some $\gamma > 1$ depending only on $n,s,q$:
\[\begin{split}
A_i \le \frac{c 2^{2si} R^{-2s} }{(l_i - l_{i-1})^{\frac{4s}{n}}} \left(1  + 2^{in} \frac{\sup_{\bar{R} \ge R} \frac{\left( \dashint_{B_{\bar{R}}} |u|^q \d x\right)^{1/q}}{\bar{R}^{2s-\eps}}}{l_i - l_{i-1}} \right) A_{i-1}^{1 + \frac{2s}{n}} \le \frac{c 2^{\gamma i} R^{-2s} }{M^{\frac{4s}{n}}} A_{i-1}^{1 + \frac{2s}{n}},
\end{split}\]
and, upon choosing $C > 0$ large enough:
\[\begin{split}
A_0 \le C^{-\frac{1}{2}} R^n M^2 \le \left( \frac{c R^{-2s} }{M^{\frac{4s}{n}}} \right)^{-\frac{n}{2s}} \left(2^{\gamma}\right)^{-\left(\frac{n}{2s}\right)^2},
\end{split}\]
By \cite[Lemma 7.1]{Giu03}, it holds $A_i \searrow 0$, which implies that $\sup_{B_{R/2}} u \le M$, and yields \eqref{eq:locbd-prelim}.\\
To prove \eqref{eq:locbd-p}, we observe that by assumption, the right hand side in \eqref{eq:locbd-prelim} is finite. Therefore, the desired result follows by standard covering and interpolation arguments based on \cite[Lemma 1.1]{GiGi82} (see e.g. \cite[Proof of Theorem 6.2]{KaWe22b} or \cite[Proof of Theorem 6.9]{Coz17}).
\end{proof}

\section{Basic properties of solutions}
\label{sec:basic}

The main result of this article (see \autoref{thm:opt-reg}), is formulated for weak solutions to the obstacle problem
\[\begin{split}
\min \{ L u , u - \phi \} = 0~~ \text{ in } \R^n,
\end{split}\]
where $\phi \in C_c^{2,\eps}(\R^n)$ for some $\eps > \max\{2s-1 , 0 \}$. 
As explained before, we note that the consideration of weak solutions to the nonlocal obstacle problem is in contrast to \cite{CRS17}, \cite{FeRo23}, and \cite{FRS23}, where viscosity solutions were analyzed. 
Thus, we need some preliminary results on weak solutions, which we provide below.


Let us define 
\[\begin{split}
H^s_{\phi}(\R^n) &:= \left\{ v \in H^s(\R^n) :v \ge \phi ~~ \text{ in } \R^n \right\},
\end{split}\]
the solution space associated to the obstacle problem. 

\begin{definition}
We say that $u$ is a weak solution to the obstacle problem 
\[\begin{split}
\min \{ L u , u- \phi \} ~~ \text{ in } \R^n,
\end{split}\]
if $u \in H^s_{\phi}(\R^n)$, and
\begin{align}
\label{eq:weak-sol}
\cE^K(u,v-u) \ge 0 ~~ \forall v \in H^s_{\phi}(\R^n).
\end{align}
\end{definition}

\begin{remark}
\label{remark:weak-sol-ex-un}
One can prove that a unique weak solution $u \in H^s_{\phi}(\R^n)$ to the obstacle problem exists. Moreover, the unique weak solution $u$ solves in the weak sense (see \autoref{def:weak-sol}):
\[\begin{split}
L u &= 0~~ \text{ in } \{ u > \phi\},\\
L u &\ge 0 ~~ \text{ in } \R^n.  
\end{split}\]
This was proved for $L = (-\Delta)^s$ in \cite{Sil07} and for more general nonlocal operators comparable to the fractional $p$-Laplacian in \cite{KKP16}. The proof in our setting goes by the same arguments, as in \cite{KKP16}.
\end{remark}

An important characterization of weak solutions to the obstacle problem is that they are the least weak supersolution above the obstacle, in the following sense:

\begin{lemma}
\label{lemma:least-super}
Assume \eqref{eq:Glower} and \eqref{eq:Gupper}. Let $u \in H^s_{\phi}(\R^n)$ be a weak solution to the obstacle problem
\[\begin{split}
\min \{ Lu , u- \phi \} = 0~~ \text{ in } \R^n.
\end{split}\]
Let $v \in H^s(\R^n)$ be a weak supersolution to 
\[\begin{split}
Lv \ge 0 ~~ \text{ in } \R^n.
\end{split}\]
Moreover, assume that $\min \{u,v\} \in H^s_{\phi}(\R^n)$. Then, $u \le v$ a.e. in $\R^n$.
\end{lemma}

\begin{proof}
Since $u$ is a weak solution to the obstacle problem, and $\min \{u,v \} \in H^s_{\phi}(\R^n)$, it holds
\[\begin{split}
\cE^K(u,\min \{u,v \} -u) \ge 0.
\end{split}\]
Moreover, since $v$ is a weak supersolution and $u - \min \{u,v \} \in H^s(\R^n)$ is nonnegative, it holds
\[\begin{split}
\cE^K(v,[u-\min \{ u,v \} ]\psi_R) \ge 0,
\end{split}\]
where, for $R > 0$, we chose $\psi_R \in C_c^{\infty}(B_{R+1})$ with $\psi \equiv 1$ in $B_R$ and $\Vert \psi_R \Vert_{C^1(\R^n)} \le 2$.
By taking the limit $R \to \infty$, an application of dominated convergence theorem, and adding the two previous lines, we obtain
\[\begin{split}
\cE^K(u-v , \min \{ u,v \} - u) \ge 0,
\end{split}\]
which implies that $|\{ u > v\}| = 0$, as desired.
\end{proof}

Let us close this section by the definition of distributional solutions to the obstacle problem.

\begin{definition}
Let $f \in L^{1}_{loc}(\R^n)$. We say that $u \in L^{\infty}_{2s-\eps}(\R^n)$ for some $\eps \in (0,2s)$ is a distributional solution to
\[\begin{split}
\min \{ L u - f , u \} = 0 ~~ \text{ in } \R^n,
\end{split}\]
if $u \ge 0$ and solves 
\[\begin{split}
L u &= f ~~ \text{ in } \{ u > 0 \},\\
L u &\ge f ~~ \text{ in } \R^n
\end{split}\]
in the distributional sense, according to \autoref{def:distr-sol}.
\end{definition}

\subsection{Semiconvexity}
 
Having at hand the characterization of a solution to the obstacle problem as the smallest supersolution lying above the obstacle (see \autoref{lemma:least-super}), we are now able to prove the semiconvexity and Lipschitz regularity of weak solutions.

\begin{lemma}
\label{lemma:semiconvexity}
Assume \eqref{eq:Glower} and \eqref{eq:Gupper}. Let $u \in H^s_{\phi}(\R^n)$ be a weak solution to the obstacle problem
\[\begin{split}
\min \{ Lu , u - \phi \} = 0 ~~ \text{ in } \R^n.
\end{split}\]
Then, the following hold true:
\begin{itemize}
\item[(i)] $u$ is Lipschitz continuous with
\[\begin{split}
\Vert u \Vert_{C^{0,1}(\R^n)} \le \Vert \phi \Vert_{C^{0,1}(\R^n)}.
\end{split}\]
\item[(ii)] $u$ is semiconvex with
\[\begin{split}
\partial_{ee}^2 u \ge - \Vert \phi \Vert_{C^{1,1}(\R^n)} ~~ \text{ in } \R^n \qquad \forall e \in \mathbb{S}^{n-1}.
\end{split}\]
\end{itemize}
\end{lemma}

\begin{proof}
The proof goes along the lines of the proof in \cite[Lemma 2.1]{CRS17}. The conclusion differs only slightly, since we work with weak solutions, instead of viscosity solutions. Indeed, to prove (i), we observe that
\[\begin{split}
v_1(x) = \Vert \phi \Vert_{L^{\infty}(\R^n)} ~~ \text{ and } ~~ v_2(x) = u(x+h) + \Vert \phi \Vert_{C^{0,1}(\R^n)}|h|, \qquad h \in \R^n,
\end{split}\]
are weak supersolutions to $L v_i \ge 0$ in $\R^n$ by \autoref{remark:weak-sol-ex-un} , and satisfy $v_i \ge \phi$.
Moreover, we have $\min \{ u,v_i \} \in H^s_{\phi}(\R^n)$. Thus, \autoref{lemma:least-super} implies that $u \le v_i$ a.e., which proves (i).\\
To prove (ii), one proceeds similar to the proof of (i), defining 
\[\begin{split}
v(x) = \frac{u(x+h) + u(x-h)}{2} + \Vert \phi \Vert_{C^{1,1}(\R^n)}|h|^2.
\end{split}\]
\end{proof}

Let us mention two direct consequences of \autoref{lemma:semiconvexity}:

\begin{remark}
Let $u$ be a weak solution to $\min \{ L u , u - \phi \} = 0$ in $\R^n$.
\begin{itemize}
\item[(i)] Then, we are able to evaluate $Lu$ in a pointwise way and we have the bound $0 \le Lu \le C$ for some $C > 0$ depending only on $n,s,\lambda,\Lambda,\Vert \phi \Vert_{C^{1,1}(\R^n)}$ (see also \autoref{lemma:distr-ptw-evaluate}).

\item[(ii)] Moreover, due to \cite[Lemma 2.2.27]{FeRo23} $u \in H^s_{\phi}(\R^n) \cap L^{\infty}(\R^n)$ is a distributional solution to $\min \{ L u , u - \phi \} = 0$ in $\R^n$ in the sense of \cite[Definition 2.2.20]{FeRo23}.
\end{itemize}
\end{remark}

\subsection{$C^{1,\tau}$ regularity of solutions}

The following lemma is a technical ingredient in some of the proofs in this article, such as the classification of blow-ups. However, we believe it to be of independent interest: It states that any solution to the obstacle problem is $C^{1,\tau}$ when $K$ satisfies \eqref{eq:Glower} and \eqref{eq:Gupper}.

\begin{proposition}
\label{thm:C1-tau}
Let $s\in(0,1)$ and $L$ be a general stable operator of the form \eqref{eq:L}-\eqref{eq:hom}-\eqref{eq:Glower}-\eqref{eq:Gupper}. 
Let $\phi \in C_c^{1,1}(\R^n)$ and $u$ be any weak solution to the obstacle problem
\[\begin{split}
\min \{ L u, u - \phi \} = 0 ~~ \text{ in } \R^n.
\end{split}\]
Then,  $u \in C^{1,\tau}(\R^n)$ and
\[\begin{split}
\Vert u \Vert_{C^{1+\tau}(\R^n)} \le C \Vert \phi \Vert_{C^{1,1}(\R^n)},
\end{split}\]
where $C > 0$ and $\tau \in (0,1)$ depend only on $n,s,\lambda,\Lambda$.
\end{proposition}

We will actually prove the following:

\begin{proposition}
\label{prop:C1-tau}
Assume \eqref{eq:Glower} and \eqref{eq:Gupper}. Let $K$ be homogeneous. Let $\alpha \in (0, s)$. Let $u \in C^{0,1}(\R^n)$ be such that for some $K > 0$:
\[\begin{split}
u &\ge 0 ~~ \text{ in } B_2,\\
D^2 u &\ge -K \mathrm{Id} ~~ \text{ in } B_2,\\
L(D_h u) &\ge -K ~~ \text{ in } \{ u > 0\} \cap B_2 \text{ in the distributional sense } ~~ \forall h \in \R^n,\\
|\nabla u| &\le K(1 + |x|^{s+\alpha}) ~~ \text{ in } \R^n.
\end{split}\]
Then, there exist $c > 0$ and $\tau \in (0,1)$ depending only on $n,s,\alpha,\lambda,\Lambda$ such that
\[\begin{split}
\Vert u \Vert_{C^{1,\tau}(B_{1/2})} \le c K.
\end{split}\]
\end{proposition}

Recall that we denote $D_h u = \frac{u(\cdot + h) - u(\cdot)}{|h|}$ for any $h \in \R^n$.

Before we prove \autoref{prop:C1-tau}, let us first state the following auxiliary result, which is reminiscent of \cite[Lemma 2.1]{CRS17}.

\begin{lemma}
\label{lemma:C1-tau-help}
Assume \eqref{eq:Glower} and \eqref{eq:Gupper}. Let $K$ be homogeneous. Then, there exist $\tau \in (0,1)$ and $\delta > 0$, depending only on $n,s,\lambda,\Lambda$, such that the following holds true:\\
Let $u \in C^{0,1}(\R^n)$ with $u(0) = 0$ be such that 
\[\begin{split}
u &\ge 0 ~~ \text{ in } B_{1/\delta},\\
D^2 u &\ge -\delta\mathrm{Id} ~~ \text{ in } B_{1/\delta},\\
L(D_h u) &\ge -\delta ~~ \text{ in } \{ u > 0\} \cap B_2 \text{ in distribution } ~~ \forall h \in \R^n,\\
\Vert \nabla u \Vert_{L^{\infty}(B_R)} &\le R^{\tau} ~~ \forall R \ge 1.
\end{split}\]
Then, there exists $c > 0$ depending only on $n,s,\lambda,\Lambda$ such that
\[\begin{split}
|\nabla u(x)| \le 2|x|^{\tau} ~~ \forall x \in \R^n.
\end{split}\]
\end{lemma}

The proof is very similar to the one in \cite[Proposition 2.2]{CRS17}. However, due to the more general class of kernels satisfying only \eqref{eq:Glower} and \eqref{eq:Gupper}, some of  the arguments need to be adapted to our situation. In particular, the set $\mathcal{C}_{\mu}$ needs to be defined appropriately.

\begin{proof}[Proof of \autoref{lemma:C1-tau-help}]
We set 
\[\begin{split}
\theta(r) = \sup_{r'\ge r} (r')^{-\tau} \sup_{B_{r'}} |\nabla u|,
\end{split}\]
and observe that $\theta(r) \le 1$ for any $r \ge 1$ by assumption. Our goal is to show that $\theta(r) \le 2$ for $r \in (0,1)$. By contradiction, we assume that $\theta(r) > 2$ for some $r \in (0,1)$, which implies that there is $r'\in (r,1)$ such that
\[\begin{split}
(r')^{-\tau} \sup_{B_{r'}}|\nabla u| \ge (1-\eps) \theta(r) \ge (1-\eps) \theta(r') \ge \frac{3}{2},
\end{split}\]
where we will choose $\eps > 0$ small enough, later. Next, we set
\[\begin{split}
\bar{u}(x) = \frac{u(r'x)}{\theta(r')(r')^{1+\tau}},
\end{split}\]
which satisfies
\begin{align}
\nonumber
\bar{u} &\ge 0 ~~ \text{ in } B_{1/\delta},\\
\nonumber
D^2 \bar{u} &\ge - (r')^{1-\tau} \delta \mathrm{Id} \ge -\delta \mathrm{Id} ~~ \text{ in } B_{1/\delta},\\
\label{eq:contradiction}
L(-D_{\bar{h}} \bar{u}) &\le (r')^{2s-1-\tau} \delta \le \delta ~~\text{ in } \{ \bar{u} > 0\} \cap B_2.
\end{align}
Moreover, by definition of $\theta,r'$:
\[\begin{split}
1 - \eps \le \sup_{|\bar{h}| \le 1/4} \sup_{B_1} (-D_{\bar{h}} \bar{u}), \qquad \sup_{|\bar{h}| \le 1/4} \sup_{B_R} (-D_{\bar{h}} \bar{u}) \le (R + 1/4)^{\tau} ~~ \forall R \ge 1.
\end{split}\]
In particular,
\begin{align}
\label{eq:C1-tau-help1}
\Vert \bar{u} \Vert_{C^{0,1}(B_{7/4})} \le 2^{\tau}.
\end{align}
Next, we take $\eta \in C_c^2(B_{3/2})$ with $\eta \equiv 1$ in $B_1$ and $\eta \le 1$ in $B_{3/2}$. Then
\[\begin{split}
1 + 2\eps \le \sup_{|\bar{h}| \le 1/4} \sup_{B_1} (-D_{\bar{h}} \bar{u} + 3 \eps \eta).
\end{split}\]
We fix $h_0 \in B_{1/4}$ and $x_0 \in B_{3/2}$ such that
\[\begin{split}
t_0 := \max_{B_{3/2}} (-D_{h_0}\bar{u} + 3 \eps \eta ) &\ge 1 + \eps,\\
-D_{h_0}\bar{u}(x_0) + 3 \eps \eta(x_0) &= t_0,
\end{split}\]
and denote $v := -D_{h_0} \bar{u}$. Then it holds, if $\tau \in (0,1)$ is small enough
\[\begin{split}
v + 3 \eps \eta &\le v(x_0) + 3 \eps \eta(x_0) = t_0 ~~ \text{ in } \overline{B_{3/2}},\\
\sup_{B_4} v &\le (4 + 1/4)^{\tau} \le 1 + \eps \le t_0,
\end{split}\]
and therefore
\[\begin{split}
v+ 3 \eps \eta \le t_0 ~~ \text{ in } \overline{B_{2}}.
\end{split}\]
Moreover, $x_0 \in \{\bar{u} > 0\}$ since otherwise $\bar{u}(x_0) - \bar{u}(x_0 - h_0) \le 0$.\\
By the same argument as in \cite{CRS17}, using $D^2 \bar{u} \ge - \delta \mathrm{Id}$, $\bar{u} \ge 0$, and $\bar{u}(0) = 0$, we obtain for $x \in B_{\frac{1}{\delta} - 1}$:
\[\begin{split}
\bar{u}(x) - \bar{u}\left(x + t \frac{x}{|x|} \right) \le \frac{\delta |x| t}{2} ~~ \forall t \in (0,1).
\end{split}\]
Combining this with \eqref{eq:C1-tau-help1}, we obtain for any $t \in (0,1)$:
\[\begin{split}
v(x) \le \frac{\bar{u}\left(x + t \frac{x}{|x|} \right) - \bar{u}(x+h_0) }{|h_0|} + \frac{\delta |x| t}{2|h_0|} &\le 2^{\tau} \left|\frac{t}{|h_0|} \frac{x}{|x|} - \frac{h_0}{|h_0|} \right| + \frac{\delta |x| t}{2|h_0|}.
\end{split}\]

Let us estimate $v(x)$ even further by making the following observation: For any $\mu \in (0,1)$, there exists $\sigma \in (0,1)$ such that for any $x \in \mathcal{C}_{\mu} := \{ x \in \R^n : \frac{x}{|x|} \cdot \frac{h_0}{|h_0|} > \mu \}$, there exists $t \in (0,|h_0|)$ such that $\left|\frac{t}{|h_0|} \frac{x}{|x|} - \frac{h_0}{|h_0|} \right| \le (1-\sigma)$. In fact, $t$ can be chosen in such a way that $\frac{t}{|h_0|}\frac{x}{|x|}$ becomes the orthogonal projection of $\frac{h_0}{|h_0|}$ onto $\{ a \frac{x}{|x|} : a \in \R\}$. Therefore,
\[\begin{split}
v(x) \le 2^{\tau} \left|\frac{t}{|h_0|} \frac{x}{|x|} - \frac{h_0}{|h_0|} \right| + \frac{\delta |x| t}{2|h_0|} \le 2^{\tau} (1- \sigma) + \frac{\delta |x|}{2}.
\end{split}\]
Moreover, we have
\[\begin{split}
1 - 2 \eps \le v(x_0) \le 1 + \eps,
\end{split}\]
and therefore
\[\begin{split}
v(x_0) - v(y) \ge 
\begin{cases}
-c \eps |y-x_0|^2 ~~  \forall y \in B_2(x_0),\\
-(|y-x_0|+2)^{\tau} + 1 - 2 \eps ~~ \forall y \in \R^n \setminus B_1(x_0),\\
-\left(2^{\tau}(1-\sigma) + \frac{3\delta M}{2} \right) + 1 - 2 \eps ~~ \forall y \in \mathcal{C}_{\mu} \cap (B_{2M}(x_0) \setminus B_M(x_0)),
\end{cases}
\end{split}\]
where $M > 0$ will be chosen large enough, later in the proof.
Consequently,
\[\begin{split}
L v(x_0) &\ge -c \eps \int_{B_2} |y-x_0|^2 K(y-x_0) \d y\\
&\quad - \int_{\R^n \setminus B_1} \left[(|y-x_0|+2)^{\tau} - 1 + 2\eps \right] K(y-x_0) \d y \\
&\quad + \int_{C_{\mu} \cap (B_{2M}(x_0) \setminus B_M(x_0))} \left[ -\left(2^{\tau}(1-\sigma) + \frac{3\delta M}{2} \right) + 1 - 2 \eps \right] K(y - x_0) \d y\\
&= I_1 + I_2 + I_3.
\end{split}\]
Clearly, by \eqref{eq:Gupper}, it holds $I_1 \ge - c \eps \to 0$, as $\eps \to 0$ and $I_2 \to 0$, as $\tau \to 0$ and $\eps \to 0$.
Finally, let us explain how to estimate $I_3$. By \autoref{lemma:K-mass}, there exists $\nu \in (0,1)$, depending only on $n,s,\lambda,\Lambda$, such that
\begin{align}
\label{eq:K-est-C1tau}
\int_{ \{ \theta \in \mathbb{S}^{n-1} : \theta\cdot \frac{h_0}{|h_0|} > \nu \}} K(\theta) \d \theta \ge c
\end{align}
for some $c > 0$ depending only on $n,s,\lambda,\Lambda$. Moreover,  let us now choose $M > 0$ so large such that for any $x_0 \in B_{3/2}$ it holds
\begin{align}
\label{eq:K-est-C1tau2}
(x_0 + \mathcal{C}_{\nu}) \cap (B_{2M}(x_0) \setminus B_{M}(x_0)) \subset C_{\nu/2} \cap (B_{2M}(x_0) \setminus B_{M}(x_0)).
\end{align}
Note that $M$ can be chosen as a uniform constant, depending only on $n$.
Thus, choosing $\mu = \frac{\nu}{2}$ and $\tau,\delta,\eps$ so small that $\left[ -\left(2^{\tau}(1-\frac{\nu}{2}) + \frac{3\delta M}{2} \right) + 1 - 2 \eps \right] \ge \frac{\nu}{4}$ for any $x \in B_{2M}(x_0) \setminus B_M(x_0)$, we can estimate, using \eqref{eq:K-est-C1tau} and \eqref{eq:K-est-C1tau2}:
\[\begin{split}
I_3 \ge \frac{\nu}{4} \int_{C_{\nu/2} \cap (B_{2M}(x_0) \setminus B_M(x_0))} \hspace{-0.6cm}  K(y - x_0) \d y \ge \frac{\nu}{4}\int_{(x_0 + C_{\nu}) \cap (B_{2M}(x_0) \setminus B_M(x_0))} \hspace{-0.6cm} K(y-x_0) \d y \ge c\frac{\nu}{4}M^{-2s}.
\end{split}\]

We have shown that there exists $c > 0$ depending only on $n,s,\lambda,\Lambda$ such that $L v(x_0) \ge c$ once $\eps,\tau,\delta$ are chosen small enough.
This is a contradiction for $\delta > 0$ small enough, since by \eqref{eq:contradiction}
\[\begin{split}
c \le L v(x_0) \le \delta.
\end{split}\]
Note that we have $c \le L v(x_0)$ in the classical sense, since $x_0$ is a local maximum of the function $v + 3 \eps \eta$, and this function has a finite tail due to the growth control that we assume on $v$. Therefore, $v+3\eps \eta$ can be touched from above by a $C^2$-function, and we deduce that $c \le Lv(x_0)$ by application of \autoref{lemma:distr-ptw-evaluate}, using that $L \eta(x_0)$ is also defined in a pointwise sense.

\end{proof}

\begin{proof}[Proof of \autoref{prop:C1-tau} and \autoref{thm:C1-tau}]
The proof of \autoref{prop:C1-tau} is a standard consequence of \autoref{lemma:C1-tau-help}, which is applied after a rescaling and truncation argument. Moreover, we need to apply the interior regularity estimates in \cite[Theorem 2.4.3]{FeRo23} to $D_h u$. Note that \cite[Theorem 2.4.3]{FeRo23} remains true under \eqref{eq:Glower} and \eqref{eq:Gupper}. For more details on this  proof, we refer to \cite{CRS17}.
\end{proof}

\section{Classification of blow-ups}
\label{sec:classification}

The goal of this section is to prove \autoref{thm:classification} about the classification of blow-ups. 

\begin{remark}
\label{remark:weak-sense}
By interior estimates (see \cite[Theorem 2.4.3]{FeRo23}), in the situation of \autoref{thm:classification} we have $\nabla u_0 \in C^{2s-\eps}_{loc}(\{ u_0 > 0\}) \cap L^{\infty}_{s+\alpha}(\R^n)$. Therefore, $L(\nabla u_0) = 0$ also holds true in the weak sense in $\{ u_0 > 0\} \cap B_R$ for any $R > 0$ by \cite[Lemma 2.2.27]{FeRo23}.
\end{remark}

The proof in \cite{CRS17} and \cite{FeRo23} is heavily based on the boundary Harnack principle in $C^1$ (or more general) domains, since it yields the uniqueness of positive solutions to $L u = 0$ in cones. However, since the full Harnack inequality fails in our setting due to the generality of the kernels under consideration, we need to come up with another argument. It turns out that a boundary Harnack principle can still be established for positive and monotone solutions outside a convex cone (see \autoref{thm:bdryHarnack-cone}). In order to prove this result, we rely on the weak Harnack inequality (see \autoref{lemma:wHI-p}) and a local boundedness estimate (see \autoref{lemma:locbd-p}), which remain true in our setup. The main challenge is to establish a certain control on the growth of the solutions in order to get rid of the nonlocal contributions in the local boundedness estimate (see \autoref{prop:L1Linfty}).

\subsection{$L^q$ growth control}

A central ingredient in the proof of a boundary Harnack principle in convex cones is the following $L^{\infty}-L^q$-estimate which differs from \eqref{eq:locbd-p} in that it only contains local quantities on both sides:

\begin{proposition}
\label{prop:L1Linfty}
Assume \eqref{eq:Glower} and \eqref{eq:Kupper-p} for some $p > \frac{n}{2s}$, and $2s < n$. Let $K$ be homogeneous. Let $\Sigma \subset \R^n$ be a closed, convex cone with non-empty interior, and vertex at $0$. Let $u \in C(\R^n) \cap L^{\infty}_{s+\alpha}(\R^n)$ for some $\alpha \in (0,s)$, and $e \in \mathbb{S}^{n-1}$ with $e \not \in \Sigma$ be such that in the distributional sense:
\[
\begin{split}
L u &= 0 ~~ \text{ in } \R^n \setminus \Sigma,\\
u &= 0 ~~ \text{ in } \Sigma,\\
u &> 0~~ \text{ in } \R^n \setminus \Sigma,\\
\partial_e u &\ge 0 ~~ \text{ in } \R^n.
\end{split}
\]
Then, there exists a constant $c > 0$, depending only on $n,s,\lambda,\Lambda,e,\Sigma,q$, such that for any $r \in (0,1)$:
\begin{align}
\label{eq:scaled-L1Linfty}
\left(\dashint_{B_r} u^q \right)^{1/q} \le \Vert u \Vert_{L^{\infty}(B_r)} \le c \left(\dashint_{B_r} u^q \right)^{1/q}. 
\end{align}
\end{proposition}

To prove \autoref{prop:L1Linfty}, we establish growth control on the $L^q$-norms (see \autoref{lemma:L1-growth-control}). The main auxiliary result for this is the following consequence of the weak Harnack inequality and a barrier argument:

\begin{lemma}
\label{lemma:dist-comp}
Assume \eqref{eq:Glower} and \eqref{eq:Gupper}. Let $K$ be homogeneous and $2s < n$. Let $\Sigma \subset \R^n$ be a closed, convex cone with non-empty interior, and vertex at $0$. Let $u,e$ be as in \autoref{prop:L1Linfty} and assume that for some $0 < q < \frac{n}{n-2s}$
\[\begin{split}
\left(\dashint_{B_1} u^q \right)^{1/q} \ge 1.
\end{split}\]
Then there exist $\eps \in (0,s)$, depending only on $n,s,\lambda,\Lambda$, and $c>0$, depending only on $n,s,\lambda,\Lambda,e,\Sigma,q$, such that
\[\begin{split}
u \ge c d^{2s-\eps} ~~ \text{ in } (\R^n \setminus \Sigma) \cap B_1,
\end{split}\]
where $d = d_{\Sigma} = \dist(\cdot,\Sigma)$.
\end{lemma}

\begin{proof}
First, note that we can find a set $D \subset \R^n$ with $\Sigma \cup (\R^n \setminus B_3) \subset D \subset \Sigma \cup (\R^n \setminus B_2)$ and satisfying the exterior ball condition. Consequently, there are $\eps \in (0,s)$ and $\rho > 0$ such that the function $d_D = \dist(\cdot,D)$ satisfies
\[\begin{split}
L(d_D^{2s-\eps}) \le -1 \le 0 \le L u ~~ \text{ in } (\R^n \setminus D) \cap \{d_D < \rho\}.
\end{split}\]
For a reference, see \cite[Lemma B.1.4]{FeRo23}, which implies that the above estimate holds true for any domain $D$ satisfying the exterior ball condition and any kernel satisfying \eqref{eq:Gupper} and \eqref{eq:Glower}.
Moreover, it holds 
\[\begin{split}
u \ge 0 = d_D^{2s-\eps}~~ \text{ in } D.
\end{split}\]
Finally, by assumption, there exist $z_0 \in B_1$, $k \in \N$, $c > 0$ independent of $u$ such that $\left(\dashint_{B_{2^{-k}}(z_0)} u^q\right)^{1/q} \ge c 2^{-kn/q}$. Moreover, since $e \not\in \Sigma$, it is possible to find $t > 0$ such that for $z = z_0 + t e \in \R^n$ it holds $d_{\Sigma}(z) \ge 1$ and by $\partial_e u \ge 0$, it holds
\[\begin{split}
\left(\dashint_{B_{2^{-k}}(z)} u^q\right)^{1/q} \ge c 2^{-kn/q}.
\end{split}\] 
Therefore, by the weak Harnack inequality (see \autoref{lemma:wHI-p}), we have for any ball $B_{1/2}(x)$ with $B_{2^{-k}(z)} \subset B_{1/2}(x) \subset (\R^n \setminus \Sigma) \cap \{d_{\Sigma} \ge \rho\}$
\[\begin{split}
\inf_{B_{1/2}(x)}u \ge c 2^{-kn/q} \left(\dashint_{B_{2^{-k}}(z)} u^q \right)^{1/q} \ge c 2^{-2kn/q}.
\end{split}\]
Note that the weak Harnack inequality is applicable since $u$ is also a weak solution in $\{d_{\Sigma} \ge \rho \}$ by interior regularity estimates (see \cite[Theorem 2.4.3, Lemma 2.2.27]{FeRo23}).\\
Now we can use the above estimate and cover the whole domain $(B_3 \setminus \Sigma) \cap \{d_{\Sigma} \ge \rho\}$ by appropriate Harnack chains until we obtain that for some $c > 0$, depending on $n,s,\lambda,\Lambda,q,e,\Sigma,\rho$, but not on $u$, it holds
\[\begin{split}
u \ge c ~~ \text{ in } B_3 \cap \{d_{\Sigma} \ge \rho\}.
\end{split}\]
In particular, this implies
\[\begin{split}
u \ge c \ge c d_D^{2s-\eps} ~~ \text{ in } \{d_D \ge \rho\}.
\end{split}\]
Therefore, the comparison principle (see \cite[Corollary 2.3.8]{FeRo23}) yields
\[\begin{split}
u \ge c d_D^{2s-\eps} = c d_{\Sigma}^{2s-\eps} ~~ \text{ in } B_1,
\end{split}\]
where we used that $d_{D} = d_{\Sigma}$ in $B_1$ by construction of $D$. This concludes the proof.
\end{proof}

The following lemma contains the growth control on the $L^q$-norms.

\begin{lemma}
\label{lemma:L1-growth-control}
Assume \eqref{eq:Glower} and \eqref{eq:Gupper}. Let $K$ be homogeneous and $2s < n$. Let $\Sigma \subset \R^n$ be a closed, convex cone with non-empty interior, and vertex at $0$. Let $u,e$ be as in \autoref{prop:L1Linfty} and assume that for some $0 < q < \frac{n}{n-2s}$
\[\begin{split}
\left(\dashint_{B_1} u^q \right)^{1/q} \le 1.
\end{split}\]
Then, there exist $\eps \in (0,s)$, depending only on $n,s,\lambda,\Lambda$, and $c>0$, depending only on $n,s,\lambda,\Lambda,e,\Sigma,q$, such that
\[\begin{split}
\left(\dashint_{B_R} u^q\right)^{1/q} \le c R^{2s-\eps} ~~ \forall R \ge 1.
\end{split}\]
\end{lemma}


\begin{remark}
By scaling, $\left(\dashint_{B_r} u^q \right)^{1/q} \le 1$ implies 
\begin{align}
\label{eq:scaled-L1-growth-control}
\left(\dashint_{B_R} u^q \right)^{1/q} \le c (R/r)^{2s-\eps} ~~ \forall R \ge r.
\end{align}
\end{remark}

\begin{proof}[Proof of \autoref{lemma:L1-growth-control}]
Let $R > 1$ and define
\[\begin{split}
u_{R}(x) = \frac{u(R x)}{\left(\dashint_{B_{R}} u^q \right)^{1/q}}.
\end{split}\]
Clearly, $u_{R}$ satisfies the assumption of \autoref{lemma:dist-comp}, and in particular $\left(\dashint_{B_1} u_{R}^q \right)^{1/q} = 1$, which implies that
\[\begin{split}
u_{R} \ge c d_{\Sigma}^{2s-\eps} ~~ \text{ in } (\R^n \setminus \Sigma) \cap B_1.
\end{split}\]
Consequently, for every $x \in (\R^n \setminus \Sigma) \cap B_1$, using $d_{\Sigma}(R x) \le R d_{\Sigma}(x)$ for $R \ge 1$:
\[\begin{split}
u(R x) \ge c d_{\Sigma}^{2s-\eps}(x) \left( \dashint_{B_{R}} u^q \right)^{1/q} \ge c R^{\eps-2s} d_{\Sigma}^{2s-\eps}(R x) \left( \dashint_{B_{R}} u^q \right)^{1/q}.
\end{split}\]
By assumption there exists $x_0 \in (\R^n \setminus \Sigma) \cap B_1 \cap \{ d_{\Sigma} > 1/100\}$ such that $u(x_0) \le c$ for some constant $c = c(n,\Sigma) > 0$. Let us choose $x = x_0 / R \in B_1$, note that $d_{\Sigma}(R x) \ge 1/100$, and deduce
\[\begin{split}
\left(\dashint_{B_{R}} u^q \right)^{1/q} \le c R^{2s-\eps} u(x_0) \le c R^{2s-\eps}.
\end{split}\] 
This proves the desired result.
\end{proof}

As a consequence of the $L^q$-growth control, we can prove \autoref{prop:L1Linfty}.

\begin{proof}[Proof of \autoref{prop:L1Linfty}]

The first estimate is trivial. To prove the second estimate, we  rescale \autoref{lemma:L1-growth-control}, obtaining \eqref{eq:scaled-L1-growth-control}, which we apply to $u / \left(\dashint_{B_r} u^q \right)^{1/q}$. This yields for any $R \ge r$:
\[\begin{split}
\left(\dashint_{B_R} u^q \right)^{1/q} \le c \left(\dashint_{B_r} u^q \right)^{1/q} (R/r)^{2s-\eps},
\end{split}\]
and therefore
\begin{align}
\label{eq:growth-control-applied}
r^{2s-\eps} \sup_{R \ge r} \frac{\left(\dashint_{B_R} u^q \right)^{1/q}}{R^{2s-\eps}} \le c \left(\dashint_{B_r} u^q \right)^{1/q},
\end{align}
which is finite since $u \in L^{\infty}_{s+\alpha}(\R^n)$, by assumption.\\
Next, we claim that
\begin{align}
\label{eq:loc-bd-apply}
\Vert u \Vert_{L^{\infty}(B_{r})} \le c \left( \dashint_{B_{2r}} u^q \right)^{1/q} + c r^{2s-\eps }\sup_{R \ge r} \frac{\left(\dashint_{B_R} u^q \right)^{1/q}}{R^{2s-\eps}}.
\end{align}
This estimate will follow from the local boundedness estimate (see \autoref{lemma:locbd-p}), however note that we cannot apply it directly, since a priori it is not clear that $u$ is a weak subsolution in $B_r$. To circumvent this issue, let us take $\delta > 0$ and define $u_{\delta} = \max\{ u,\delta \}$. Clearly, by interior regularity estimates (see \cite[Theorem 2.4.3, Lemma 2.2.27]{FeRo23}), $u \in C^{2s}_{loc}(B_1 \cap (\R^n \setminus \Sigma))$ is a weak solution to $Lu = 0$ in $\Omega$ for any $\Omega \Subset B_1 \cap (\R^n \setminus \Sigma)$. Therefore, $u_{\delta}$ is a weak subsolution to $Lu_{\delta} \le 0$ in $\Omega$ due to \autoref{lemma:truncation-subsol}. By continuity of $u$ in $B_1 \cap (\R^n \setminus \Sigma)$, it follows in particular that $u_{\delta}$ is a weak subsolution in a neighborhood of $\{u \ge \delta\} \cap B_1$. Since $u_{\delta} \equiv \delta$ in $\{u < \delta\} \cap B_1$, $u_{\delta}$ is a weak subsolution to $Lu_{\delta} \le 0$ in $B_1$. Moreover, note that $u_{\delta} \in L^{\infty}_{2s-\eps}(\R^n)$ due to \eqref{eq:growth-control-applied}. Thus, by \eqref{eq:locbd-p} in \autoref{lemma:locbd-p}:
\[\begin{split}
\Vert u \Vert_{L^{\infty}(B_r)} \le \Vert u_{\delta} \Vert_{L^{\infty}(B_r)} &\le c \left( \dashint_{B_{2r}} u_{\delta}^q \right)^{1/q} + c r^{2s-\eps }\sup_{R \ge r} \frac{\left(\dashint_{B_R} u_{\delta}^q \right)^{1/q}}{R^{2s-\eps}}\\
&\le c \left( \dashint_{B_{2r}} u^q \right)^{1/q} + c r^{2s-\eps }\sup_{R \ge r} \frac{\left(\dashint_{B_R} u^q \right)^{1/q}}{R^{2s-\eps}} + c\delta,
\end{split}\]
which proves \eqref{eq:loc-bd-apply} upon taking the limit $\delta \to 0$.\\
Therefore, by a combination of \eqref{eq:loc-bd-apply} and \eqref{eq:growth-control-applied}:
\[\begin{split}
\Vert u \Vert_{L^{\infty}(B_{r})} \le c \left( \dashint_{B_{2r}} u^q \right)^{1/q} + c r^{2s-\eps }\sup_{R \ge r} \frac{\left(\dashint_{B_R} u^q \right)^{1/q}}{R^{2s-\eps}} \le c \left( \dashint_{B_r} u^q \right)^{1/q},
\end{split}\]
as desired.
\end{proof}

\subsection{A boundary Harnack principle in convex cones}

The main auxiliary result in the proof of the classification of blow-ups (see \autoref{thm:classification}) is the following boundary Harnack principle in convex cones:

\begin{theorem}
\label{thm:bdryHarnack-cone}
Assume \eqref{eq:Glower} and \eqref{eq:Kupper-p} for some $p > \frac{n}{2s}$, and $2s < n$. Let $K$ be homogeneous. Let $\Sigma \subset \R^n$ be a closed, convex cone with non-empty interior, and vertex at $0$. Let $u,v \in C(\R^n) \cap L^{\infty}_{s+\alpha}(\R^n)$ for some $\alpha \in (0,s)$ be distributional solutions to
\[\begin{split}
\begin{cases}
L u = 0 = L v ~~ \text{ in } \R^n \setminus \Sigma,\\
u = 0 = v ~~ \text{ in } \Sigma,\\
u,v > 0 ~~ \text{ in } \R^n \setminus \Sigma.
\end{cases}
\end{split}\]
Moreover, assume that there exist $e,e' \in \mathbb{S}^{n-1}$ with 
$e,e'\not \in \Sigma$ such that 
\[\begin{split}
\partial_e u, \partial_{e'} v \ge 0.
\end{split}\] 
and assume that $u$ and $v$ are normalized 
\[\begin{split}
\left(\dashint_{B_1} u^q \right)^{1/q} = 1 = \left(\dashint_{B_1} v^q \right)^{1/q}.
\end{split}\]
Then, there exists a constant $c > 0$, depending only on $n,s,\lambda,\Lambda,e,e',\Sigma,q$, such that
\[\begin{split}
c^{-1} u \le v \le c u~~ \text{ in } B_{1/2}.
\end{split}\]
\end{theorem} 


The main scheme of the proof follows the one in \cite{RoSe19}.

\begin{proof}[Proof of \autoref{thm:bdryHarnack-cone}]
By assumption, the following normalization condition is satisfied.
\begin{align}
\label{eq:normalization}
\left(\dashint_{B_1} u^q \right)^{1/q} = 1 = \left(\dashint_{B_1} v^q \right)^{1/q}.
\end{align}
First, we claim 
\begin{align}
\label{eq:bdry-help-1}
u \le C, ~~ v \le C ~~ \text{ in } B_{3/4}.
\end{align}
This is an immediate consequence of \autoref{prop:L1Linfty} (resp. of \eqref{eq:scaled-L1Linfty}) and \eqref{eq:normalization}, which yield that
\[\begin{split}
\sup_{B_{3/4}} u \le C \left(\dashint_{B_{3/4}} u^q\right)^{1/q} \le C.
\end{split}\]
Moreover, using \eqref{eq:normalization}, we can deduce from \autoref{lemma:dist-comp} 
\begin{align}
\label{eq:bdry-help-2}
u \ge c > 0, ~~ v \ge c > 0 ~~ \text{ in } (B_1 \setminus \Sigma) \cap \{ d_{\Sigma} \ge \rho \},
\end{align}
where $\rho > 0$ is a constant, which we are allowed to (and will) choose small enough in the sequel. Note that the constants $c,C > 0$ might depend on $n,s,\lambda,\Lambda,e,e',\Sigma,\rho$, but not on $u,v$.\\
Having at hand \eqref{eq:bdry-help-1} and \eqref{eq:bdry-help-2}, we can follow the strategy of the proof of the boundary Harnack principle in \cite{RoSe19} (resp. in \cite[Theorem 4.3.2]{FeRo23}):\\
We define $\xi \in C_c^{\infty}(B_{2/3})$ with $0 \le \xi \le 1$ and $\xi \equiv 1$ in $B_{1/2}$. Moreover, we fix $\mu = 1/100$ and take $\eta \in C^{\infty}((B_{3/4 + \mu} \setminus B_{3/4 - \mu}) \cap \{ d > \rho\})$ with $0 \le \eta \le 1$ and $\eta \equiv 1$ in $(B_{3/4 + \mu/2} \setminus B_{3/4 - \mu/2}) \cap \{ d > \rho\}$. With these definitions at hand, we introduce
\[\begin{split}
w := u \mathbbm{1}_{B_{3/4}} + C_1 (\xi - 1) + C_2 \eta.
\end{split}\]
Our goal is to prove that $w \le c v$ in $\R^n$ by using the comparison principle for distributional solutions.

First, if $C_1 > 0$ is large enough, we get by \eqref{eq:bdry-help-1}:
\[\begin{split}
w \le 0 \le v ~~ \text{ in } \R^n \setminus (B_{2/3} \cup \supp(\eta)),
\end{split}\]
and for $C_2 > 0$ large enough we get
\[\begin{split}
L w &= L u - L (u \mathbbm{1}_{\R^n \setminus B_{3/4}}) + C_1 L \xi + C_2 L \eta\\
&\le C + C_1 C - C_2 c \le -1 ~~ \text{ in } (\R^n \setminus \Sigma) \cap B_{2/3}.
\end{split}\]
Let us explain how to obtain the latter estimate. The following
\begin{align}
\label{eq:operator-tail}
L (u \mathbbm{1}_{\R^n \setminus B_{3/4}}) \le C \sup_{R \ge 1} \frac{\left( \dashint_{B_R} |u|^q \right)^{1/q}}{R^{2s-\eps}} \le C ~~ \text{ in } B_{2/3}
\end{align}
can be established using \autoref{lemma:tail-est-p}, the $L^q$-growth control (see \autoref{lemma:L1-growth-control}), and the normalization condition \eqref{eq:normalization}. Moreover, the estimate $L \xi \le C$ follows from \eqref{eq:Kupper-p}.
To prove that $L \eta \le -c$ in $(\R^n \setminus \Sigma) \cap B_{2/3}$, we argue as follows:
For any $z \in (\R^n \setminus \Sigma) \cap B_{2/3}$, it holds
\[\begin{split}
L \eta(z) = - \int_{\R^n} \eta(y) K(z-y) \d y \le - \int_{(B_{3/4 + \mu/2} \setminus B_{3/4 - \mu/2}) \cap \{ d_{\Sigma} > \rho\}} K(z-y) \d y,
\end{split}\]
since $\supp(\eta) \cap B_{2/3} = \emptyset$. We need to argue that the right hand side is bounded by a constant depending only on $n,s,\lambda,\Lambda$. In fact, by \autoref{lemma:K-mass}, there exists $\delta_0 > 0$ depending only on $\lambda,\Lambda$ such that for any $e \in \mathbb{S}^{n-1}$ 
\[\begin{split}
- \int_{ (B_{3/4 + \mu/2} \setminus B_{3/4 - \mu/2}) \cap \left\{\frac{z-y}{|z-y|} \cdot e > \delta_0 \right\} } K(z-y) \d y \le - c
\end{split}\]
for a uniform constant $c > 0$, depending only on $n,s,\lambda,\Lambda$, but not on $z$. Moreover, since $\Sigma$ is convex, it must be contained in a half-space. Consequently, there exists $e \in \mathbb{S}^{n-1}$ such that $\{x \cdot e > 0\} \subset \R^n \setminus \Sigma $, and therefore, for $\rho > 0$ small enough, depending on $\delta_0$, it holds 
\[\begin{split}
(B_{3/4 + \mu/2} \setminus B_{3/4 - \mu/2}) \cap  \left\{ y : \frac{z-y}{|z-y|} \cdot e \ge \delta_0 \right\} \subset (B_{3/4 + \mu/2} \setminus B_{3/4 - \mu/2}) \cap  \{y : d_{\Sigma}(y) > \rho\}.
\end{split}\]
Therefore, we have
\[\begin{split}
L \eta(z) &\le - \int_{(B_{3/4 + \mu/2} \setminus B_{3/4 - \mu/2}) \cap \{ d_{\Sigma} > \rho\}} K(z-y) \d y \\
&\le - \int_{ (B_{3/4 + \mu/2} \setminus B_{3/4 - \mu/2}) \cap \left\{\frac{z-y}{|z-y|} \cdot e > \delta_0 \right\} } K(z-y) \d y \le - c,
\end{split}\]
as desired.\\
All in all, we obtain
\[\begin{split}
L w  \le -1 \le 0 = L v ~~ \text{ in } (\R^n \setminus \Sigma) \cap B_{2/3}.
\end{split}\]
Moreover, by \eqref{eq:bdry-help-1} and \eqref{eq:bdry-help-2} we have
\[\begin{split}
w \le C_3 v ~~ \text{ in } \supp(\eta).
\end{split}\]
Next, by construction
\[\begin{split}
w \le 0 = v ~~ \text{ in } \Sigma.
\end{split}\]
Therefore, by the comparison principle (see \cite[Corollary 2.3.8]{FeRo23}),
\[\begin{split}
w \le C_3 v ~~ \text{ in } \R^n.
\end{split}\]
Now we are in the position to conclude the proof. In fact,
 since $w \equiv u$ in $B_{1/2}$, the aforementioned estimate implies
\[\begin{split}
u \le C_3 v ~~ \text{ in } B_{1/2}.
\end{split}\]
By changing the roles of $u,v$, we deduce the desired result.
\end{proof}

The main ingredient in the proof of \autoref{thm:classification} is the following result about the global comparability of positive solution in closed, convex cones, reminiscent of \cite[Theorem 3.1]{CRS17}:

\begin{corollary}
\label{thm:uniqueness}
Assume \eqref{eq:Glower} and \eqref{eq:Kupper-p} for some $p > \frac{n}{2s}$, and $2s < n$. Let $K$ be homogeneous. Let $\Sigma \subset \R^n$ be a closed, convex cone with non-empty interior, and vertex at $0$. Let $u,v \in C(\R^n) \cap L^{\infty}_{s+\alpha}(\R^n)$ for some $\alpha \in (0,s)$ be distributional solutions to
\[\begin{split}
\begin{cases}
L u = 0 = L v ~~ \text{ in } \R^n \setminus \Sigma,\\
u = 0 = v ~~ \text{ in } \Sigma,\\
u,v > 0 ~~ \text{ in } \R^n \setminus \Sigma.
\end{cases}
\end{split}\]
Moreover, assume that there exist $e,e' \in \mathbb{S}^{n-1}$ with 
$e,e'\not \in \Sigma$ such that 
\[\begin{split}
\partial_e u, \partial_{e'} v \ge 0.
\end{split}\]
Then, there exists $A > 0$ such that
\[\begin{split}
A^{-1} u \le v \le A u ~~ \text{ in } \R^n.
\end{split}\]
\end{corollary}

The proof of \autoref{thm:uniqueness} follows directly from \autoref{thm:bdryHarnack-cone} by a scaling argument:

\begin{proof}[Proof of \autoref{thm:uniqueness}]
We define for $R \ge 2$:
\[\begin{split}
u_R(x) := \frac{u(Rx)}{\left(\dashint_{B_R} u^q \right)^{1/q}}, \qquad v_R(x) := \frac{v(Rx)}{\left(\dashint_{B_R} v^q \right)^{1/q}},
\end{split}\]
and observe that by construction $\left(\dashint_{B_1} u_R^q \d x\right)^{1/q} = \left(\dashint_{B_1} v_R^q \d x \right)^{1/q} = 1$. Therefore, by application of \autoref{thm:bdryHarnack-cone}, we obtain
\[\begin{split}
c^{-1} u_R \le v_R \le c u_R ~~ \text{ in } B_{1/2}
\end{split}\]
for some $c > 0$, depending on $n,s,\lambda,\Lambda,e,e',\Sigma,q$, but not on $u,v,R$. In particular, we have
\[\begin{split}
c^{-1} \frac{u(x)}{\left(\dashint_{B_R} u^q \right)^{1/q}} \le \frac{v(x)}{\left(\dashint_{B_R} v^q \right)^{1/q}} \le c \frac{u(x)}{\left(\dashint_{B_R} u^q \right)^{1/q}} ~~ \forall x \in B_{R/2},
\end{split}\]
which, upon fixing an arbitrary point $x \in B_1 \setminus \Sigma$,  implies that the quotient $\left(\dashint_{B_R} u^q \right)^{1/q} / \left(\dashint_{B_R} v^q \right)^{1/q}$ is uniformly positive and uniformly bounded as $R \to \infty$. Thus, we obtain
\[\begin{split}
A^{-1} u(x) \le v(x) \le A u(x) ~~ \forall x \in B_{R/2}
\end{split}\]
for some $A > 0$, independent of $A$, which implies the desired result upon taking the limit $R \to \infty$.
\end{proof}

\subsection{Proof of \autoref{thm:classification}}

Before we prove the classification of blow-ups \autoref{thm:classification}, we consider the special case that $\{ u_0 = 0\}$ is a closed convex cone with non-empty interior:

\begin{lemma}
\label{lemma:cone-classification}
The statement of \autoref{thm:classification} holds true in case $\{u_0 = 0\} = \Sigma$ is a closed, convex cone with non-empty interior, and vertex at $0$.
\end{lemma}

Before we start the proof, let us observe that since by assumption $K \in L^p(\mathbb{S}^{n-1})$, there exists $\Lambda > 0$ such that  $K$ satisfies \eqref{eq:Kupper-p}. Moreover, since $K \neq 0$, there is $0 < \lambda \le \Lambda$ such that $K$ satisfies \eqref{eq:Glower}.

\begin{proof}
Note that in case $n = 1$, the desired result follows directly from  \cite[Theorem 1.10.15]{FeRo23}. Thus, we can assume from now on that $n \ge 2$ and therefore $2s < n$.\\
Since $\Sigma \subset \R^n$ is a closed convex cone with non-empty interior and vertex at $0$, there exist $n$ linearly independent vectors $e_1, ... , e_n \in \mathbb{S}^{n-1}$ such that $- e_i \in \Sigma$, while $e_i \not\in\Sigma$. We claim that the linear space satisfies
\begin{align}
\label{eq:1d}
\mathrm{dim} \left\{ \sum_{i = 1}^n \lambda_i \partial_{e_i} u_0 : (\lambda_1, \dots , \lambda_n) \in \R^n \right\} \le 1.
\end{align}
Let us take any $1 \le i,j \le n$ with $i \neq j$ and assume that $\partial_{e_i} u_0 \not\equiv 0 \not\equiv \partial_{e_j} u_0$ (otherwise, \eqref{eq:1d} follows). Note that \eqref{eq:1d} holds true if we can show that there is $M > 0$ such that
\begin{align}
\label{eq:partial-uniqueness}
\partial_{e_i} u_0 = M \partial_{e_j} u_0 ~~ \text{ in } \R^n.
\end{align}
First, we observe that it holds in the distributional sense
\[\begin{split}
\begin{cases}
L \partial_{e_i} u_0 = L \partial_{e_j} u_0 = 0 ~~ \text{ in } \R^n \setminus \Sigma,\\
\partial_{e_i} u_0 = \partial_{e_j} u_0 = 0~~ \text{ in } \Sigma,\\
\partial_{e_i} u_0 , \partial_{e_j} u_0 > 0 ~~ \text{ in } \R^n \setminus \Sigma,\\
\partial_{e_i}(\partial_{e_i} u_0) , \partial_{e_j}(\partial_{e_j} u_0)  \ge 0 ~~ \text{ in } \R^n,
\end{cases}
\end{split}\]
where the third property holds true by the weak Harnack inequality (see \autoref{lemma:wHI-p}) since $\partial_{e_i} u_0, \partial_{e_j} u_0 \ge 0$ (by convexity and $-e_i,-e_j \in \Sigma$), and $\partial_{e_i} u_0 \not \equiv 0 \not\equiv \partial_{e_j} u_0$. Note that \autoref{lemma:wHI-p} is applicable due to \autoref{remark:weak-sense}. The fourth property follows from the convexity of $u_0$.
Moreover, note that $\partial_{e_i} u_0, \partial_{e_j} u_0$ are continuous due to \autoref{prop:C1-tau}. Therefore, \autoref{thm:uniqueness} is applicable to $\partial_{e_i} u_0, \partial_{e_j} u_0$, which implies that there is $A > 0$ such that
\[\begin{split}
A^{-1} \partial_{e_i} u_0 \le \partial_{e_j} u_0 \le A\partial_{e_i} u_0 ~~ \text{ in } \R^n.
\end{split}\]

To prove \eqref{eq:partial-uniqueness}, let us define
\[\begin{split}
\kappa^{\ast} := \sup \{\kappa > 0 : \partial_{e_i} u_0 \ge \kappa \partial_{e_j} u_0 ~~ \text{ in } \R^n\}, \qquad w := (\partial_{e_i} - \kappa^{\ast}\partial_{e_j}) u_0 \ge 0,
\end{split}\]
and assume by contradiction that $w \not\equiv 0$. Note that by definition of $w$, we have $w = \partial_{e_i - \kappa^{\ast}e_j} u_0$, and therefore 
\[\begin{split}
L w &= 0 ~~ \text{ in } \R^n \setminus \Sigma,\\
w &= 0~~ \text{ in } \Sigma,\\
w &> 0 ~~ \text{ in } \R^n \setminus \Sigma,\\
\partial_{e_i - \kappa^{\ast}e_j} w &\ge 0 ~~ \text{ in } \R^n.
\end{split}\]
The first and second property follow by the assumptions on the blow-up $u_0$. The third property follows from the weak Harnack inequality since $w \ge 0$ by definition and $w \not \equiv 0$. The fourth property follows by convexity of $u_0$. \\
We distinguish now between the following two cases:
\[\begin{split}
\text{either } e_i - \kappa^{\ast}e_j \not\in\Sigma, ~~ \text{ or } ~~ e_i - \kappa^{\ast}e_j \in \Sigma.
\end{split}\]
In case $e_i - \kappa^{\ast}e_j \not\in\Sigma$ we can apply \autoref{thm:uniqueness} to $w,\partial_{e_j}u_0$, which yields that there is $B > 0$ such that
\[\begin{split}
B^{-1} w \le \partial_{e_j} u_0 \le B w ~~ \text{ in } \R^n.
\end{split}\]
Therefore, 
\[\begin{split}
\partial_{e_i} u_0 \ge (\kappa^{\ast} + B^{-1})\partial_{e_j} u_0 ~~ \text{ in } \R^n,
\end{split}\]
which contradicts the definition of $\kappa^{\ast}$. Consequently, we must have $w \equiv 0$, which yields that $\partial_{e_i} u_0 = \kappa^{\ast} \partial_{e_j} u_0$ and proves \eqref{eq:partial-uniqueness}, and therefore also \eqref{eq:1d}.

On the other hand, in case $e_i - \kappa^{\ast}e_j \in\Sigma$, let us show that it must be
\begin{align}
\label{eq:half-space}
\{t (e_i - \kappa^{\ast}e_j) : t \in \R\} \subset \Sigma.
\end{align}
In fact, if we had $-(e_i - \kappa^{\ast}e_j) \not\in \Sigma$, then   $- t(e_i - \kappa^{\ast}e_j) \not\in \Sigma$ and thus $w(- t(e_i - \kappa^{\ast}e_j)) > 0$ for every $t > 0$, while $t(e_i - \kappa^{\ast}e_j) \in \Sigma$ and thus $w(t(e_i - \kappa^{\ast}e_j)) = 0$ for any $t > 0$. This contradicts $\partial_{e_i - \kappa^{\ast}e_j} w \ge 0$. Therefore, we have $\pm(e_i - \kappa^{\ast}e_j) \in \Sigma$, which yields \eqref{eq:half-space}.\\
By convexity of $u_0$, \eqref{eq:half-space} implies that $0 = \partial_{e_i - \kappa^{\ast}e_j} u_0 = \partial_{e_i} u_0 - \kappa^{\ast}\partial_{e_j} u_0$ (see \cite[Lemma 4.3]{CRS17}). However, this in turn yields \eqref{eq:partial-uniqueness}, and therefore also \eqref{eq:1d}, as desired.

We have established \eqref{eq:1d}, which proves that there exist $1 \le k \le n$ and $\kappa_i \ge 0$ such that for any $1 \le i \le n$ with $i \not= k$ it holds $\partial_{e_i - \kappa_i e_k} u_0 = 0$ in $\R^n$, which implies that $u_0$ is invariant in $n-1$ directions, i.e. that there exist $e \in \mathbb{S}^{n-1}$ and $\phi \in C^1(\R)$ such that $u_0(x) = \phi(x \cdot e)$. In particular, $\Sigma = \{ e \cdot x \le 0 \}$ since $0 \in \partial \Sigma$.\\
Next, due to \cite[Lemma B.1.5]{FeRo23}, we have $(-\Delta)^s_{\R} \phi'(x \cdot e) = 0$ in $\{x \cdot e > 0\}$ and it holds $\phi'(x \cdot e) = 0$ in $\{x \cdot e \le 0\}$. Moreover, $\phi'(t) \le C(1+ t^{s+\alpha})$. Thus, we can apply \cite[Theorem 1.10.15]{FeRo23} and obtain that $\phi'(x \cdot e) = a(x \cdot e)_+^s$ and therefore
\[\begin{split}
u_0(x) = \phi(x \cdot e) = \frac{a}{1+s} (x \cdot e)_+^{1+s}, ~~ \text{ for some } a \ge 0,
\end{split}\] 
which proves the desired result in case $\Sigma = \{u_0 = 0\}$ is a closed convex cone in $\R^n$ with non-empty interior and vertex at $0$.
\end{proof}

Now, we are in the position to give the 

\begin{proof}[Proof of \autoref{thm:classification}]
First, let us assume that $u_0 \not\equiv 0$, since otherwise there is nothing to prove. Let us observe that $\{u_0 = 0\} \subset \R^n$ is a convex set with $0 \in \partial \{u_0 = 0\}$.

By proceeding as in the proof of \cite[Proposition 4.4.3]{FeRo23}, we can find a sequence $R_m \nearrow \infty$ such that
\[\begin{split}
u_m(x) = \frac{u_0(R_m x)}{R_m \Vert \nabla u_0 \Vert_{L^{\infty}(B_{R_m})}}
\end{split}\]
satisfies in the distributional sense
\[\begin{split}
L(D_h u_m) &\ge 0 ~~ \text{ in } \{ u_m > 0\} = R_m^{-1} \{u_0 > 0\},\\
\Vert \nabla u_m \Vert_{L^{\infty}(B_R)} &\le 2R^{s+\alpha} ~~ \forall R \ge 1, \qquad \Vert \nabla u_m \Vert_{L^{\infty}(B_1)} = 1.
\end{split}\]
Moreover, by convexity of $u_m$, and the $C^{1,\tau}$-estimates from \autoref{prop:C1-tau}, they converge locally uniformly (up to a subsequence) to a function $u_{\infty} \in C^{0,1}(\R^n)$
satisfying
\[\begin{split}
u_{\infty} &\ge 0, ~~ \text{ and } D^2 u_{\infty} \ge 0 ~~ \text{ in } \R^n,\\
L(\nabla u_{\infty}) &= 0, ~~ \text{ and } L(D_h u_{\infty}) \ge 0 ~~ \text{ in } \R^n \setminus \Sigma ~~ \text{ in distributional sense},\\
\Vert \nabla u_{\infty} \Vert_{L^{\infty}(B_R)} &\le 2R^{s+\alpha} ~~ \forall R \ge 1, \qquad \Vert \nabla u_{\infty} \Vert_{L^{\infty}(B_1)} \ge 1,
\end{split}\]
where $\Sigma = \cap_{m} R_m^{-1} \{u_0 = 0\}$ denotes the limiting closed convex cone of the blow-down sequence. Note that 
we used the stability of distributional solutions (see \cite[Proposition 2.2.31]{FeRo23}).

In case $\Sigma$ is a closed, convex cone with empty interior, we claim that in the distributional sense
\begin{align}
\label{eq:global-distr-sol}
L(D_h u_{\infty}) = 0 ~~ \text{ in } ~~ \R^n, \qquad \forall h \in \R^n.
\end{align} 
From here, by the Liouville theorem with growth for distributional solutions (see \cite[Corollary 2.4.13]{FeRo23}), it turns out that $u_{\infty} = 0$, which is a contradiction.  Therefore, we can rule out that $\Sigma$ is a closed, convex cone with empty interior. \\
Let us prove that \eqref{eq:global-distr-sol} holds true:
First, note that there exists $e \in \mathbb{S}^{n-1}$ such that $\Sigma \subset \{x \in \R^n : e \cdot x = 0 \}$.
Moreover, observe that $D_h u_{\infty} \in C^{2s+1-\eps}_{loc}(\R^n \setminus \Sigma)$ by application of interior regularity estimates (see \cite[Theorem 2.4.2]{FeRo23}) to $\nabla u_{\infty}$, and using also the growth control on $\nabla u_{\infty}$. Therefore, $L(D_h u_{\infty}) \ge 0$ in a pointwise sense in $\R^n \setminus \Sigma$. 
Next, let us introduce $\phi(x) = \exp(-|e \cdot x|^{1-\theta})$ for some $\theta \in (\max\{0,1-2s \}, 1-s)$.
Then, according to \cite[Lemma B.1.1 and Lemma B.1.2]{FeRo23}, we have $L \phi \ge -C$ in $\R^n$, and moreover $L \phi = + \infty$ in $\Sigma$ in a pointwise way. The proof carries over to our more general class of operators since $\phi$ is one-dimensional (see the proof of \cite[Lemma 2.5.2]{FeRo23} and use \eqref{eq:Gupper}). Moreover, note that since $u_{\infty} \in C^{0,1}(\R^n)$, the function $\phi_{\eps} = D_h u_{\infty} + \eps \phi$ has a positive cusp on $\{x \cdot e = 0\}$ and hence
also satisfies $L \phi_{\eps} = + \infty$ in $\Sigma$ in a pointwise sense. Thus
\begin{align}
\label{eq:global-distr-sol-approx}
L \phi_{\eps} \ge -C \eps ~~ \text{ in } \R^n
\end{align}
pointwise, and by \autoref{lemma:ptw-distr}, \eqref{eq:global-distr-sol-approx} also holds true in the distributional sense. Moreover, note that since $\phi_{\eps} \to D_h u_{\infty}$ in $L^1_{loc}(\R^n)$, we have
\[\begin{split}
\int_{\R^n} (L \eta) \phi_{\eps} \d x \to \int_{\R^n} (L \eta) D_h u_{\infty} \d x
\end{split}\]
by the same arguments as in the proof of \cite[Proposition 2.2.31]{FeRo23}, and therefore we have in the distributional sense
\[\begin{split}
L(D_h u_{\infty}) \ge 0 ~~ \text{ in } ~~ \R^n, \qquad \forall h \in \R^n.
\end{split}\]

By taking $-h$ instead of $h$, and employing the same arguments as before, we obtain that $L(-D_{-h} u_{\infty}) \ge 0$ in $\R^n$, which yields \eqref{eq:global-distr-sol}, as desired. Thus, as explained before, the case where $\Sigma$ is a closed, convex cone with empty interior cannot happen.


Alternatively, the limiting set $\Sigma$ is a closed convex cone with non-empty interior and we have
\begin{align}
\label{eq:Sigma-limit}
\Sigma = \{ u_{\infty} = 0 \} = \cap_{m} R_m^{-1} \{u_0 = 0\}.
\end{align}
In that case, by \autoref{lemma:cone-classification}, we obtain
\[\begin{split}
u_{\infty}(x) = a_{\infty} (x \cdot e)_+^{1+s}, ~~ \text{ for some } a_{\infty} \ge 0.
\end{split}\]
In particular, $\Sigma = \{x \in \R^n : x \cdot e \le 0 \}$. Thus, due to \eqref{eq:Sigma-limit}, it must be $\{u_0 = 0\} = \{x \in \R^n : x \cdot e \le 0 \}$, so also $\{u_0 = 0\}$ is a closed convex cone with non-empty interior. An application of \autoref{lemma:cone-classification} to $u_0$ yields
\[\begin{split}
u_{0}(x) = a_{0} (x \cdot e)_+^{1+s}, ~~ \text{ for some } a_{0} \ge 0,
\end{split}\]
which concludes the proof.
\end{proof}

\subsection{Quantitative closeness to the blow-up}

As a direct consequence of the classification of blow-ups, we have the following quantitative estimate on closeness of a solution to the obstacle problem to the blow-up. This result is a counterpart of \cite[Proposition 4.4.14]{FeRo23} (see also \cite[Theorem 2.2]{FRS23}).

\begin{corollary}
\label{lemma:prop4.4.14}
Assume \eqref{eq:Glower} and \eqref{eq:Kupper-p} for some $p > \frac{n}{2s}$. Let $K$ be homogeneous, $\alpha \in (0,\min \{ s, 1-s \} )$ and let $\eps_0 > 0$ and $R_0 > 1$. Then, there is $\eta > 0$, depending only on $n,s,\lambda,\Lambda,\alpha,\eps_0,R_0$, such that the following holds true:\\
Let $u \in C^{0,1}(\R^n)$ such that
\begin{itemize}
\item[(i)] $\min \{ Lu - f , u \} = 0$ in $\R^n$ in the distributional sense, where $|\nabla f| \le \eta$,
\item[(ii)] $u \ge 0$ and $D^2 u \ge - \eta \mathrm{Id}$ in $\R^n$ with $0 \in \partial\{ u > 0 \}$,
\item[(iii)] $\Vert \nabla u \Vert_{L^{\infty}(B_R)} \le R^{s+\alpha}$ for any $R \ge 1$.
\end{itemize}
Then, it holds
\begin{align}
\label{eq:close-to-blowup}
\Vert u - \kappa (x \cdot e)_+^{1+s} \Vert_{C^{0,1}(B_{R_0})} \le \eps_0
\end{align}
for some $e \in \mathbb{S}^{n-1}$ and $\kappa \ge 0$.
\end{corollary}

\begin{proof}
We assume by contradiction that there is no $\eta > 0$ for which the result holds. Then, there are $\eta_k \to 0$, $L_k$ satisfying \eqref{eq:Glower}, \eqref{eq:Kupper-p}, $f_k, u_k$ satisfying (i),(ii),(iii) with $\eta = \eta_k$ but violating \eqref{eq:close-to-blowup} for any $e \in \mathbb{S}^{n-1}$ and $\kappa \ge 0$. By \autoref{prop:C1-tau}, $u_k$ locally converges in $C^{1,\tau}$ to a limiting $u_0 \in C^{1,\tau}$ up to a subsequence. By the stability of distributional solutions (see \cite[Proposition 2.2.31]{FeRo23}), there is $L$ satisfying \eqref{eq:Glower}, \eqref{eq:Kupper-p} such that $u_0$ satisfies (i),(ii),(iii) with $\eta = 0$. Thus, by \autoref{thm:classification}, it must be $u_0(x) = \kappa(x \cdot e)_+^{1+s}$ for some $\kappa,e$, a contradiction.
\end{proof}


\section{Regularity of the free boundary}
\label{sec:regularity}

The main result of this article are the regularity of the free boundary near regular points and the optimal $C^{1,s}$-regularity of solutions. They are summarized in \autoref{thm:opt-reg}, which we will prove in this section.\\
As we explained before, the main tool in the proof of \autoref{thm:opt-reg}, once the classification of blow-ups is established,  is the quantitative estimate \autoref{prop:4.4.15}, which we will prove first.

\subsection{Lipschitz regularity}

As a first step towards proving \autoref{prop:4.4.15}, we establish Lipschitz regularity of the free boundary near regular points. In fact, we will prove slightly more, namely that the free boundary is  Lipschitz with an arbitrarily small constant. Let us give the following definition:

\begin{definition}[Lipschitz domain]
Let $\rho_0 > 0$. We say that a domain $\Omega \subset \R^n$ is a Lipschitz domain in $B_{\rho_0}$ with constant less than $\delta$, if there are $g : \R^{n-1} \to \R$, $e \in \mathbb{S}^{n-1}$ such that 
\[\begin{split}
\Omega \cap B_{\rho_0} = \{ \bar{x}_n > g(\bar{x}_1,\bar{x}_2,\dots,\bar{x}_{n-1} )  \} \cap B_{\rho_0},
\end{split}\]
where $\bar{x} = R x$ for some rotation $R$ with $R e = e_n$, and 
\[\begin{split}
\Vert g \Vert_{C^{0,1}(B_{\rho_0})} \le \delta.
\end{split}\]
\end{definition}

The following is the main result of this subsection. It states that the free boundary is a Lipschitz domain with a small constant once the solution is close enough to the blow-up. Such result was already known for operators whose kernel is comparable to the fractional Laplacian (see \cite[Lemma 4.4.13 and Proposition 4.4.15]{FeRo23}). In our case, the proof needs to be modified in order to deal with kernels that are possibly degenerate in some directions.

\begin{lemma}
\label{lemma:4.4.13}
Assume \eqref{eq:Glower} and \eqref{eq:Gupper}. Let $K$ be homogeneous and $\alpha \in (0,\min \{ s , 1-s\})$. For any $\kappa_0 > 0$, $\rho_0 > 0$, and $\delta < \rho_0$ there are $\eps > 0$, $R_0 > 1$ depending only on $\lambda,\Lambda,n,s,\delta,\rho_0,\kappa_0$ such that the following holds true:\\
Let $u \in C^{0,1}(\R^n)$ be such that
\begin{itemize}
\item[(i)] $\min \{ Lu - f , u \} = 0$ in $B_{R_0}$, in the distributional sense, where $|\nabla f| \le \eps$,
\item[(ii)] $0 \in \partial\{ u > 0 \}$, and $D^2 u \ge - \eps \mathrm{Id}$ in $B_{R_0}$,
\item[(iii)] $\Vert \nabla u \Vert_{L^{\infty}(B_R)} \le R^{s+\alpha}$ for any $R \ge R_0$,
\item[(iv)] $\Vert u - \kappa (x \cdot e)_+^{1+s} \Vert_{C^{0,1}(B_{R_0})} \le \eps$ for some $\kappa \ge \kappa_0$ and $e \in \mathbb{S}^{n-1}$.
\end{itemize}
Then, for any $e' \in \mathbb{S}^{n-1}$ with $e \cdot e' \ge \delta$, it holds
\begin{align}
\label{eq:monotonicity-cone}
\partial_{e'} u \ge 0 ~~ \text{ in } B_{\rho_0},
\end{align}
and $\{ u > 0\} \cap B_{\rho_0}$ is Lipschitz with constant less than $c \delta$ for some $c = c(n) > 0$.\\
Moreover
\begin{align}
\label{eq:e-der-lower-bound}
\partial_e u \ge c (\delta \rho_0)^{s} ~~ \text{ in } \{ x \in \R^n : x \cdot e \ge \delta \rho_0 \} \cap B_{R_0} \subset \{ u > 0\}
\end{align}
for some $c > 0$, depending only on $\kappa_0,\eps,\delta$.
\end{lemma}

\begin{proof}
We define $u_0(x) = \kappa (x \cdot e)_+^{1+s}$ and observe that for any $e' \in \mathbb{S}^{n-1}$, we have $\partial_{e'} u_0(x) = (1+s)\kappa (e \cdot e')(x \cdot e)_+^s$. Therefore, whenever $\delta > 0$ and $e \cdot e' > \delta$, it holds
\[\begin{split}
\partial_{e'} u_0 \ge 0 ~~ \text{ in } \R^n, \qquad \partial_{e'} u_0 \ge \kappa_0 \delta^{1+s} ~~ \text{ in } \{x \cdot e  \ge \delta \}.
\end{split}\]
Let us choose $\eps \le \frac{\kappa_0}{2} \delta^{1+s}$. Then,  since
\begin{align}
\label{eq:der-closeness}
|\partial_{e'} u - \partial_{e'} u_0 | \le \eps ~~ \text{ in } B_{R_0},
\end{align}
we obtain that $w := \mathbbm{1}_{B_{R_0}}\partial_{e'} u$ satisfies
\[\begin{split}
|L w| &\le \eps + cR_0^{\alpha-s} ~~ \text{ in } B_{R_0/2} \setminus E ~~ \text{ in the distributional sense},\\
w &\equiv 0 ~~ \text{ in } E \cup (\R^n \setminus B_{R_0}),\\
w &\ge - \eps ~~ \text{ in } B_{R_0},\\
w &\ge \frac{\kappa_0}{2} \delta^{1+s} ~~ \text{ in } \{ x \cdot e \ge \delta \} \cap B_{R_0},
\end{split}\]
where we set $E = \{ u = 0\}$. The first property follows from (iii), \eqref{eq:Gupper} and \autoref{lemma:tail-est-p}, and the stability of distributional solutions (see \cite[Proposition 2.2.31]{FeRo23}) applied to $D_h u$. The second property follows from \autoref{prop:C1-tau}.

Now we turn to the actual proof of \eqref{eq:monotonicity-cone}. We claim that for any $\rho_0 > 0$ and $\delta < \rho_0$, we can find $R_0 > 1$ and $\eps  > 0$ such that $w \ge 0$ in $B_{\rho_0}$, where $R_0,\eps$ depend only on $\lambda,\Lambda,n,s$, $\delta$, $\rho_0$, and $\kappa_0$.\\
To prove the claim, we follow the proof of \cite[Lemma 4.4.13]{FeRo23}:
First, we choose a radial bump function $\psi \in C_c^{\infty}(B_2)$ with $\psi \equiv 1$ in $B_1$, $0 \le \psi \le 1$, and set for $t > 0$
\[\begin{split}
\psi_t(x) = - \eps - t + \eps \psi(x / \rho_0).
\end{split}\] 
Let us assume that the claim is false. In that case, there exist $z \in B_{\rho_0} \setminus E$ and $t > 0$ such that $\psi_t$ touches $w$ from below at $z$. Note that therefore, by \autoref{lemma:distr-ptw-evaluate} it holds $Lw(z) \ge -\eps$ in the pointwise sense ($\partial_{e'} u$ is H\"older-continuous by \autoref{prop:C1-tau}). Clearly, we have the following estimate
\[\begin{split}
L(w-\psi_t)(z) \ge Lw(z) - |L \psi_t(z)| \ge -\eps - c R_0^{\alpha-s} - c_2 \eps \rho_0^{-2s}.
\end{split}\]

Note that we used \eqref{eq:Gupper} and the properties of $\psi$ in order to estimate $|L \psi_t(z)| \le c_2 \eps \rho_0^{-2s}$. 
On the other hand, we have, using that $(w-\psi_t)(z) = 0$, $w \ge \psi_t$ and $\psi_t \le 0$:
\[\begin{split}
L(w - \psi_t)(z) &= - \int_{\R^n} (w-\psi_t)(x) K(z-x) \d x\\
&\le - \int_{\{ x \cdot e \ge \delta \} \cap B_{R_0}} (w-\psi_t)(x) K(z-x) \d x\\
&\le - \frac{\kappa_0}{2}\delta^{1+s}  \int_{\{ x \cdot e \ge \delta \} \cap B_{R_0}} K(z-x) \d x\\
&\le - \frac{\kappa_0}{2}\delta^{1+s} C,
\end{split}\]
where $C = C(\kappa_0,\delta,\lambda,\Lambda,s,n,\rho_0) > 0$.
Let us explain how to prove the last estimate. First of all, due to \autoref{lemma:K-mass}, there exists $\delta_0 > 0$ depending only on $n,s,\lambda,\Lambda$ such that for any $r  > 0$
\[\begin{split}
\int_{\left\{\frac{z-x}{|z-x|} \cdot e \ge \delta_0 \right\} \cap (B_{2r}(z) \setminus B_r(z))} K(z-x) \d x \ge c r^{-2s}.
\end{split}\]
Let us take any $\delta \le \rho_0$. Clearly, there exists $r > 0$ depending only on $\rho_0, \delta_0$ such that for any $z \in B_{\rho_0}$ it holds:
\[\begin{split}
 \left\{ \frac{z-x}{|z-x|} \cdot e \ge \delta_0 \right\} \cap ( B_{2r}(z) \setminus B_r(z)) \subset \{ x \cdot e \ge \delta \} \cap ( B_{2r}(z) \setminus B_r(z)).
\end{split}\]
Upon choosing $R_0$ so large that $B_{2r}(z) \subset B_{R_0}$ (this choice only depends on $\rho_0$), we obtain 
\[\begin{split}
- \frac{\kappa_0}{2}\delta^{1+s}  \int_{\{ x \cdot e \ge \delta \} \cap B_{R_0}} K(z-x) \d x & \leq  - \frac{\kappa_0}{2}\delta^{1+s} \int_{\left\{\frac{z-x}{|z-x|} \cdot e \ge \delta_0 \right\} \cap (B_{2r}(z) \setminus B_r(z))} K(z-x) \d x\\
&\le - \frac{c\kappa_0}{2}\delta^{1+s} r^{-2s},
\end{split}\]
as desired.\\
Having at hand the two-sided estimate for $L(w-\psi_t)(z)$, we obtain a contradiction upon taking the limit $R_0 \to \infty$ and $\eps \to 0$. This implies \eqref{eq:monotonicity-cone}, as claimed.

Thus, by \eqref{eq:monotonicity-cone}, $\{ u > 0\} \cap B_{\rho_0}$ is a Lipschitz epigraph in direction $e$ with Lipschitz constant bounded by $c \delta$. 
Indeed, to see this, we can follow the proof of \cite[Proposition 4.4.15]{FeRo23} and obtain
\[\begin{split}
u &= 0 ~~ \text{ in } \Sigma_-,\\
u &> 0 ~~ \text{ in } \Sigma_+,
\end{split}\]
where 
\[\begin{split}
\Sigma_{\pm} = \{x \in B_{\rho_0} : x = x_0 \pm t \tau, ~~ \tau \in \mathbb{S}^{n-1}, ~~\tau \cdot e \ge \delta, ~~ t > 0 \}
\end{split}\]
for any $x_0 \in B_{\rho_0} \cap \partial \{ u > 0\}$. This implies that $\partial \{ u > 0\} \cap B_{\rho_0}$ satisfies the interior and exterior cone condition with explicit cones $\Sigma_{\pm}$, and therefore $\partial \{ u > 0\} \cap B_{\rho_0}$ is Lipschitz with constant bounded by $c \delta$.

The last claim, namely \eqref{eq:e-der-lower-bound}, follows from the observation that for $x \in \{ x \in \R^n :  x \cdot e \ge \delta \rho_0 \} \cap B_{R_0}$ it holds
\[\begin{split}
\partial_e u_0(x) = (1+s)\kappa (x \cdot e)_+^s \ge \kappa_0 (\delta \rho_0)^s.
\end{split}\]
Moreover, by choosing $\eps < \frac{\kappa_0}{2} (\delta \rho_0)^s$ (making it smaller, if necessary) and \eqref{eq:der-closeness}, we obtain
\[\begin{split}
\partial_e u(x) \ge \kappa_0 (\delta \rho_0)^s - \eps \ge \frac{\kappa_0}{2} (\delta \rho_0)^s,
\end{split}\]
as desired.
\end{proof}

\subsection{Uniform non-degeneracy near the free boundary}

We already know that the free boundary is Lipschitz with a small enough constant. In order to prove the $C^{1,\gamma}$-regularity of the free boundary (see \autoref{prop:4.4.15}), we are lacking  control on the non-degeneracy of the solutions close to the free boundary. This property is established in \autoref{lemma:lower-dist-bound}. A key ingredient in its proof is the existence of suitable barrier functions (see \cite{RoSe17}):

\begin{lemma}
\label{lemma:C1-barrier}
Assume \eqref{eq:Glower} and \eqref{eq:Gupper}. Let $K$ be homogeneous. Then, for every $\theta \in (0,s)$ there is $\eta > 0$, depending only on $n,s,\lambda,\Lambda,\theta$, such that the functions
\[\begin{split}
\Phi(x) = \left( e \cdot x - \eta |x| \left[1 -  \left(\frac{x}{|x|} \cdot e \right)^2 \right] \right)^{s+\theta}_+,\\
\Psi(x) = \left( e \cdot x + \eta |x| \left[1 -  \left(\frac{x}{|x|} \cdot e \right)^2 \right] \right)^{s-\theta}_+
\end{split}\]
satisfy for some constant $c > 0$, depending only on $n,s,\lambda,\Lambda,\theta$,
\[\begin{split}
\begin{cases}
L \Phi &\le -c d^{\theta - s} ~~ \text{ in } \mathcal{C}_{\eta},\\
\Phi &= 0 ~~ \text{ in } \R^n \setminus \mathcal{C}_{\eta},
\end{cases} \qquad\qquad
\begin{cases}
L \Psi &\ge c d^{-\theta - s} ~~ \text{ in } \mathcal{C}_{-\eta},\\
\Psi &= 0 ~~ \text{ in } \R^n \setminus \mathcal{C}_{-\eta},
\end{cases}
\end{split}\]
where 
\[\begin{split}
\mathcal{C}_{\pm \eta} = \left\{ x \in \R^n : \frac{x}{|x|} \cdot e > \pm \eta  \left[ 1 - \left(\frac{x}{|x|} \cdot e \right)^2 \right] \right\}.
\end{split}\]
\end{lemma}

The barriers $\Phi, \Psi$ have been introduced in \cite{RoSe17}. Since the cones $\mathcal{C}_{\pm \eta}$ merely have a Lipschitz boundary, one cannot expect the corresponding barriers to have homogeneity $s$, however, for any $\theta \in (0,1)$, one can find barriers with homogeneity $s \pm \theta$ if the cones $\mathcal{C}_{\pm \eta}$ are close enough to a half-space, i.e. have a small enough Lipschitz constant.\\
In the sequel, we shortly explain how the proof of \cite[Lemma 4.1]{FeRo23} can be adapted to general kernels satisfying only \eqref{eq:Glower} and \eqref{eq:Gupper}.

\begin{proof}
We only explain how to prove the result for $\Phi$. By homogeneity, it is enough to prove that $L \Phi \le -c$ on points belonging to $e + \partial \mathcal{C}_{\eta}$ for some $c > 0$. Given $P \in \partial \mathcal{C}_{\eta}$, we define
\[\begin{split}
\Phi_{P,\eta}(x) := \Phi(P+e+x) = (1 + e \cdot x - \eta \phi_P(x))_+^{s+\theta},
\end{split}\]
where $\phi_P$ is a function satisfying $\phi_P(0) = 0$, $|\nabla \phi_P(x)| \le C$ for any $x \in \R^n \setminus \{-P-e\}$, $|D^2 \phi_P(x)| \le C$ for any $x \in B_{1/2}$, $|\phi_P(x)| \le c |x|$ for $|x| \ge 1/2$. Moreover, note that 
\begin{align}
\label{eq:C2-convergence}
\Vert \Phi_{P,\eta} - (1+e \cdot)_+^{s+\theta} \Vert_{C^{2}(B_{1/2})} \to 0.
\end{align}
Note that the proof is complete, once we show that $L \Phi_{P,\eta}(0) \le -c$. This property follows, once we verify the following two properties:
\begin{align}
\label{eq:C1-barrier-help1}
L(\Phi_{P,\eta})(0) &\to L((1+e \cdot)_+^{s+\theta})(0), ~~ \text{ as } \eta \searrow 0,\\
\label{eq:C1-barrier-help2}
L((1+e \cdot)_+^{s+\theta})(0) &\le -c,
\end{align}
where $c = c(s,\theta,\lambda,\Lambda) > 0$.
To prove \eqref{eq:C1-barrier-help1},
\[\begin{split}
|L(\Phi_{P,\eta})(0) &- L((1+e \cdot)_+^{s+\theta})(0)| = \left|\int_{\R^n} [\Phi_{P,\eta}(y) - (1 + e \cdot y)_+^{s+\theta}] K(y) \d y \right| \\
&\le \Vert \Phi_{P,\eta} - (1+e \cdot)_+^{s+\theta} \Vert_{C^{2}(B_{1/2})} \int_{B_{1/2}} |y|^2 K(y) d y + C \int_{\R^n \setminus B_{1/2}} |\eta \phi_P(y)|^{s+\theta} K(y) \d y\\
&\le \Vert \Phi_{P,\eta} - (1+e \cdot)_+^{s+\theta} \Vert_{C^{2}(B_{1/2})} \int_{B_{1/2}} |y|^2 K(y) d y + C \eta^{s+\theta} \int_{\R^n \setminus B_{1/2}} |y|^{s+\theta} K(y) \d y\\
&\le C(\Vert \Phi_{P,\eta} - (1+e \cdot)_+^{s+\theta} \Vert_{C^{2}(B_{1/2})} + \eta^{s+\theta})\\
&\to 0 ~~ \text{ as } \eta \to 0,
\end{split}\]
where we used \eqref{eq:C2-convergence}, and in the last step we applied \eqref{eq:Gupper} and \autoref{lemma:tail-est-p}.
To prove \eqref{eq:C1-barrier-help2}, we use the symmetry and homogeneity of $K$ to compute 
\[\begin{split}
L((1+e\cdot)_+^{s+\theta})(0) \le - c,
\end{split}\]
for some $c > 0$, depending only on $n,s,\lambda,\Lambda$, following the arguments in \cite[Lemma B.1.6]{FeRo23}.
\end{proof}

We are now in the position to establish the non-degeneracy close to the free boundary.

\begin{lemma}
\label{lemma:lower-dist-bound}
Assume \eqref{eq:Glower} and \eqref{eq:Gupper}. Let $K$ be homogeneous, $\alpha \in (0,\min \{ s , 1-s\})$, and $\theta \in (0,\alpha)$.
For any $\kappa_0 > 0$ there is $\rho > 1$, depending only on $n,s,\alpha,\lambda,\Lambda,\kappa_0,\theta$, such that for any $\rho_0 > \rho$, there are $\eps > 0$, $R_0 > 1$, depending only on $\lambda,\Lambda,n,s,\rho_0,\kappa_0,\theta$, such that the following holds true:\\
Let $u \in C^{0,1}(\R^n)$ be such that
\begin{itemize}
\item[(i)] $\min \{ Lu - f , u \} = 0$ in $B_{R_0}$ in the distributional sense, where $|\nabla f| \le \eps$,
\item[(ii)] $0 \in \partial\{ u > 0 \}$, and $D^2 u \ge - \eps \mathrm{Id}$ in $B_{R_0}$,
\item[(iii)] $\Vert \nabla u \Vert_{L^{\infty}(B_R)} \le R^{s+\alpha}$ for any $R \ge R_0$,
\item[(iv)] $\Vert u - \kappa (x \cdot e)_+^{1+s} \Vert_{C^{0,1}(B_{R_0})} \le \eps$ for some $\kappa \ge \kappa_0$ and $e \in \mathbb{S}^{n-1}$.
\end{itemize}
Then, we have
\begin{align}
\label{eq:lower-dist-bound}
\partial_e u \ge c d^{s+\theta} ~~ \text{ in } B_{\rho_0}
\end{align}
for some constant $c > 0$, depending only on $n,s,\lambda,\Lambda,\kappa_0,\theta$.
\end{lemma}

\begin{proof}
Let $\theta \in (0,1)$ be given and $\eta$ be as in \autoref{lemma:C1-barrier}.
By \autoref{lemma:4.4.13}, for any $\rho_0 > 0$ and $\delta < \rho_0$ there are $\eps < 1$, $R_0 > 1$, depending only on $\lambda,\Lambda,n,s,\rho_0,\kappa_0,\theta$, such that for any $x_0 \in \partial \{ u > 0\} \cap B_{\rho_0}$
\begin{align}
\label{eq:cone-contained}
(x_0 + \mathcal{C}_{\eta}) \cap B_{2\rho_0}(x_0) \subset \{ u > 0 \} \cap B_{4\rho_0}.
\end{align}
Moreover, it holds
\[\begin{split}
\partial_e u &\ge 0 ~~ \text{ in } B_{4 \rho_0},\\
\partial_e u &\ge c_1 (\delta \rho_0)^{s} ~~ \text{ in } \{ x \in \R^n : x \cdot e \ge \delta \rho_0 \} \cap B_{4 \rho_0} \subset \{ u > 0\}
\end{split}\]
for some $c_1 > 0$ depending only on $\kappa_0$. Moreover we observe that (iii) and (iv) imply
\begin{align}
\label{eq:three-and-four}
\Vert \nabla u \Vert_{L^{\infty}(B_R)} \le C_0 R^{s+\alpha} ~~ \forall R \ge 1
\end{align}
for some $C_0 > 0$, depending only on $\kappa, s$.
Next, we define $v = u\mathbbm{1}_{B_{4\rho_0}}$. We have for any $x_0 \in \partial \{ v > 0 \} \cap B_{\rho_0/4}$, using \eqref{eq:cone-contained}, \eqref{lemma:tail-est-p}, as well as the previous three displays: 
\[\begin{split}
|L \partial_e v|& \le (\eps + C_1\rho_0^{\alpha-s}) ~~ \text{ in } (x_0 + \mathcal{C}_{\eta}) \cap B_{\rho_0},\\
\partial_e v &\ge 0 ~~ \text{ in } B_{4\rho_0},\\
\partial_e v &\ge c_1 (\delta \rho_0)^s ~~ \text{ in } \{x \in \R^n : x \cdot e \ge \delta \rho_0 \} \cap B_{4\rho_0},
\end{split}\]
where $C_1 > 0$ depends only on $n,s,\kappa,\lambda,\Lambda$. 
Moreover, note that the PDE in the first property holds true in the distributional sense, and by \autoref{prop:C1-tau}, we know that $\partial_e v \in C(\overline{B_{\rho_0}})$.\\
In particular, given any $\rho_0 > 1$ we can make $\delta \rho_0$ so small (by choosing $\delta = c \rho_0^{-1}$ for some small enough $c > 0$ depending on $\eta,\lambda,\Lambda$) such that 
\[\begin{split}
\partial_e v &\ge c_2 ~~ \text{ in } (x_0 + \mathcal{C}_{\eta}) \cap (B_{4\rho_0} \setminus B_{\rho_0}),\\
c_2 &\le \int_{ \{x \in \R^n : x \cdot e \ge \delta \rho_0 \} \cap (B_{4\rho_0} \setminus B_{3\rho_0})}\hspace{-0.8cm}  K(z-y) \d y = -L ( \mathbbm{1}_{\{x \in \R^n : x \cdot e \ge \delta \rho_0 \} \cap (B_{4\rho_0} \setminus B_{3\rho_0})})(z)  ~~ \forall z \in (x_0 + \mathcal{C}_{\eta}) \cap B_{\rho_0},
\end{split}\]
where $c_2 > 0$ depends only on $n,s,\lambda,\Lambda,\kappa_0,\eta$, and we used \autoref{lemma:K-mass} in order to establish the second property. 
Now, we define
\[\begin{split}
\pi(x) = \frac{\Phi(x-x_0)}{\Vert \Phi(\cdot -x_0) \Vert_{L^{\infty}(B_{4\rho_0})}} \mathbbm{1}_{B_{4\rho_0}}(x) + C_3 \mathbbm{1}_{\{x \in \R^n : x \cdot e \ge \delta \rho_0 \} \cap (B_{4\rho_0} \setminus B_{3\rho_0})}(x),
\end{split}\]
and observe that by \autoref{lemma:C1-barrier}
\[\begin{split}
L \pi &= \frac{L (\Phi(x-x_0) \mathbbm{1}_{B_{4 \rho_0}})}{\Vert \Phi(\cdot -x_0) \Vert_{L^{\infty}(B_{4\rho_0})}} + C_3 L ( \mathbbm{1}_{\{x \in \R^n : x \cdot e \ge \delta \rho_0 \} \cap (B_{4\rho_0} \setminus B_{3\rho_0})}) \\
&\le c_3 - c_2 C_3 \\
&\le -1 ~~ \text{ in } (x_0 + \mathcal{C}_{\eta}) \cap B_{\rho_0},
\end{split}\]
upon choosing $C_3 > 1$ large enough depending only on $n,s,\lambda,\Lambda$. \\
We define $\Pi(x) = c_2 \pi(x)/(1+C_3)$, and choose $\rho_0 > 1$ large enough (largeness depends on $n,s,\lambda,\Lambda,\kappa_0,\eta$) (and making $\eps < 1$ smaller, if necessary) such that $(\eps + C_1\rho_0^{\alpha-s}) \le c_2 /(1+C_3)$. Then, we have the following properties:
\[\begin{split}
L \Pi(x) &\le - c_2 /(1+C_3) \le - (\eps + C_1\rho_0^{\alpha-s}) \le L \partial_e v(x) \qquad \forall x \in (x_0 + \mathcal{C}_{\eta}) \cap B_{\rho_0},\\
\Pi(x) &\le c_2 \le \partial_e v(x) \qquad \forall x \in (x_0 + \mathcal{C}_{\eta}) \cap (B_{4\rho_0} \setminus B_{\rho_0}),\\
\Pi(x) &= 0 = \partial_e v(x) \qquad \forall x \in (x_0 + \mathcal{C}_{\eta}) \cap (\R^n \setminus B_{4\rho_0}),\\
\Pi(x) &= 0 \le \partial_e v(x) \qquad  \forall x \in  [\R^n \setminus (x_0 + \mathcal{C}_{\eta})] \setminus [\{ x \in \R^n : x \cdot e \ge \delta \rho_0 \} \cap (B_{4 \rho_0} \setminus B_{3\rho_0})],\\
\Pi(x) &\le c_2 \le \partial_e v(x) \qquad \forall x \in [\R^n \setminus (x_0 + \mathcal{C}_{\eta})] \cap [\{x \in \R^n : x \cdot e \ge \delta \rho_0 \} \cap (B_{4\rho_0} \setminus B_{3\rho_0})].
\end{split}\]
Altogether, we deduce from the comparison principle for continuous distributional solutions (see \cite[Corollary 2.3.8]{FeRo23}) that 
\[\begin{split}
\partial_e v(x) \ge \Pi(x),
\end{split}\]
which in particular means that for any $t \in (0,1)$ by the homogeneity of $\Phi$:
\[\begin{split}
\partial_e v(x_0 + t e) \ge c t^{s+\theta},
\end{split}\]
where $c > 0$ depends on $n,s,\lambda,\lambda,\kappa_0,\eta$, which implies \eqref{eq:lower-dist-bound}. 
This concludes the proof.
\end{proof}

\subsection{Proof of \autoref{prop:4.4.15}}

%
%
%

The goal of this section is to prove \autoref{prop:4.4.15}. First, we establish the following $C^{s-\theta}$-estimate up to the boundary, which holds true for domains with sufficiently small Lipschitz constants. This result is reminiscent of \cite[Lemma 5.2]{RoSe17}.

\begin{lemma}
\label{lemma:reg-up-to-Lipschitz}
Assume \eqref{eq:Glower} and \eqref{eq:Gupper}. Let $K$ be homogeneous. Let $\theta \in (0,s)$ and $K_0 > 0$.
Let $\delta \in (0,1)$ and $\Omega \subset \R^n$ be a Lipschitz domain in $B_1$ with constant less than $\delta$. Let $v \in C(\overline{B_1})$ be a distributional solution to
\[\begin{split}
|L v| &\le K_0 ~~ \text{ in } \Omega \cap B_1,\\
v &= 0 ~~ \text{ in } B_1 \setminus \Omega,
\end{split}\]
and assume that
\[\begin{split}
\Vert v \Vert_{L^{\infty}(B_R)} \le K_0 R^{2s-\theta} ~~ \forall R \ge 1.
\end{split}\]
Then, there is $\delta_0 > 0$, depending only on $n,s,\lambda,\Lambda,\theta$, such that if $\delta \le \delta_0$:
\[\begin{split}
\Vert v \Vert_{C^{s-\theta}(B_{1/2})} \le CK_0,
\end{split}\]
where $C > 0$ depends only on $n,s,\lambda,\Lambda,\theta$.
\end{lemma}

\begin{proof}
First, let us consider $u = v \mathbbm{1}_{B_2}$ and observe that (after a normalization) $u$ satisfies
\[\begin{split}
|L u| &\le 1 ~~ \text{ in } \Omega \cap B_1,\\
u &= 0 ~~ \text{ in } B_1 \setminus \Omega,\\
u &= 0 ~~ \text{ in } \R^n \setminus B_2,\\
\Vert u \Vert_{L^{\infty}(\R^n)} &\le 1.
\end{split}\]
We claim that for any $x_0 \in \partial \Omega \cap B_{1/2}$ it holds
\begin{align}
\label{eq:reg-up-to-Lipschitz-help1}
|u(x)| \le C |x-x_0|^{s-\theta} ~~ \text{ in } \Omega \cap B_{1/4}(x_0).
\end{align}
To prove \eqref{eq:reg-up-to-Lipschitz-help1}, let us take $\eta \in (0,1)$ as in \autoref{lemma:C1-barrier}, and observe that by the assumption on $\Omega$, if the Lipschitz constant is small enough, depending on $\eta$, it holds 
\[\begin{split}
B_{1/2}(x_0) \cap \Omega \subset B_{1/2}(x_0) \cap (x_0 + \mathcal{C}_{-\eta}).
\end{split}\]
Moreover, let us define $\Psi$ as the function in \autoref{lemma:C1-barrier} with $2\eta$ and set
\[\begin{split}
\Pi = C\Psi(\cdot - x_0) + \mathbbm{1}_{B_2(x_0) \setminus B_{1/2}(x_0)},
\end{split}\]
where we choose $C > 0$ large enough, such that
\begin{align}
\label{eq:reg-up-to-Lipschitz-help2}
L \Pi \ge d_{x_0 + \mathcal{C}_{-\eta}}^{-\theta-s} \ge 1  ~~ \text{ in } B_{1/4}(x_0) \cap \Omega.
\end{align}
Note that by \eqref{eq:Gupper} it is easily seen that $L \mathbbm{1}_{B_2(x_0) \setminus B_{1/2}(x_0)} \ge -c$ in $B_{1/4}(x_0)$ for some constant $c > 0$, depending only on $n,s,\lambda,\Lambda$, and therefore one can find $C > 0$ satisfying \eqref{eq:reg-up-to-Lipschitz-help2}.\\
Then, it holds
\[\begin{split}
L \Pi \ge 1 \ge L u ~~ \text{ in } B_{1/4}(x_0) \cap \Omega.
\end{split}\]
Moreover, note that 
\[\begin{split}
\Psi(x) \ge c |x|^{s-\theta} ~~ \forall x \in \mathcal{C}_{-\eta},
\end{split}\]
where $c > 0$ depends on $\eta$. Therefore, we have upon choosing $C > 0$ large enough:
\[\begin{split}
\Pi \ge 1 \ge u ~~ \text{ in }( B_{1/2}(x_0) \setminus B_{1/4}(x_0)) \cap \Omega.
\end{split}\]
Moreover, by construction (and since $B_{1/2}(x_0) \setminus \Omega \subset B_1 \setminus \Omega$)
\[\begin{split}
\Pi &\ge 0 = u ~~ \text{ in } (B_{1/2}(x_0) \setminus \Omega) \cup (\R^n \setminus B_2(x_0)),\\
\Pi &\ge 1 \ge u ~~ \text{ in } B_{2}(x_0) \setminus B_{1/2}(x_0).
\end{split}\]
All in all, we can apply the comparison principle (see \cite[Corollary 2.3.8]{FeRo23}) and obtain
\[\begin{split}
u \le \Pi ~~ \text{ in } \R^n.
\end{split}\]
In particular, since $\Pi(x) = \Phi(x - x_0) \le c |x-x_0|^{s - \theta}$ for any $x \in \Omega \cap B_{1/2}(x_0)$, we obtain \eqref{eq:reg-up-to-Lipschitz-help1} (after repeating all the aforementioned arguments with $-\Pi$), as desired.\\
Next, we observe that by interior estimates (see \cite[Theorem 2.4.3]{FeRo23}) combined with a standard rescaling argument (see \cite[Proof of Proposition 2.5.4]{FeRo23}) it holds for any $x \in \Omega \cap B_{1/2}$
\[\begin{split}
[ u ]_{C^{s -\theta}(B_{d(x)/2}(x))} \le C d^{\theta-s}(x)(\Vert u \Vert_{L^{\infty}(B_{d(x)}(x))} + d^{s-\theta}(x)) \le C
\end{split}\]
for some constant $C > 0$, where we used \eqref{eq:reg-up-to-Lipschitz-help1} in the last step. Combining this estimate with \eqref{eq:reg-up-to-Lipschitz-help1} yields the desired result.
\end{proof}

\begin{remark}
\label{remark:almost-opt}
In the situation of \autoref{prop:4.4.15} it holds $u \in C^{1+s-\tilde{\eps}}$ for any $\tilde{\eps} \in (0,1)$ if $\eps > 0$ is small enough (depending on $\tilde{\eps}$). This follows by application of \autoref{lemma:reg-up-to-Lipschitz} with $v = \nabla u$, which is possible due to \autoref{prop:C1-tau}.
\end{remark}

Next, we observe that the pointwise boundary regularity estimate \cite[Proposition 5.4]{RoSe17} also holds true in our setup. This is because its proof is merely based on blow-up arguments and the Liouville theorem on a half-space, which remain true for general stable operators.

\begin{lemma}[see Proposition 5.4 in \cite{RoSe17}]
\label{lemma:proposition5.4}
Assume \eqref{eq:Glower} and \eqref{eq:Gupper}. Let $K$ be homogeneous. Let $\alpha \in (0,s)$ and $C_0 \ge 1$. Let $\delta \in (0,1)$ and $\Omega \subset \R^n$ be a Lipschitz domain in $B_{1/\delta}$ with constant less than $\delta$.


Then, there is $\delta_0 > 0$ such that the following holds true for any $\delta < \delta_0$: If $u,v \in L^{\infty}_{s+\alpha}(\R^n)$ are distributional solutions to
\[\begin{split}
L v_1, L v_2 &\le \delta ~~ \text{ in } B_{1/\delta} \cap \Omega,\\
v_1 = v_2 &= 0 ~~ \text{ in } B_{1/\delta} \setminus \Omega,\\
\Vert v_1 \Vert_{L^{\infty}(B_R)} &+ \Vert v_2 \Vert_{L^{\infty}(B_R)} \le C_0 R^{s+\alpha} ~~ \forall R \ge 1,
\end{split}\]
and 
\[\begin{split}
v_2 \ge 0 ~~ \text{ in } B_1, \qquad C_0^{-1} \le \sup_{B_1} v_2 \le C_0.
\end{split}\]
Then, there is $K \in \R$ with $|K| \le C$ such that
\[\begin{split}
|v_1(x) - K v_2(x)| \le C |x|^{s+\alpha} ~~ \forall x \in B_1,
\end{split}\]
where $C > 0$ depends only on $\delta,C_0,\alpha,s,\Lambda,\lambda,n$.
\end{lemma}

\begin{proof}
The proof goes exactly as in \cite[Proposition 5.4]{RoSe17}. \\
Note that although the statement in \cite{RoSe17} is written for a so-called improving Lipschitz domain $\Omega$, this property is never used in the proof. Instead, it suffices for $\Omega$ to be Lipschitz in a large enough ball with a small enough Lipschitz constant. This can be achieved by proving the main estimate (5.14) in \cite{RoSe17} for $\delta > 0$ small enough (where $\delta$ must be at least so small that \cite[Lemma 5.3]{RoSe17} is applicable). As in \cite{RoSe17} one proves (5.14) by contradiction, assuming that there are sequences $\Omega_j$ with Lipschitz parameter $\delta_j \to 0$ such that (5.14) fails. Then, clearly $\Omega_j$ (and therefore also $j \Omega_j$) converges to a half-space, as $j \to \infty$, which allows to carry out the proof as in \cite{RoSe17}. See also \cite[Proposition 5.1]{Tor23} for an analogous proof for flat Lipschitz domains in the parabolic setting

Moreover, the statement in \cite{RoSe17} and all the aforementioned lemmas in \cite{RoSe17} are written for equations with an unbounded right hand side. In our setting, we only require the result with a bounded right hand side, however the proofs remain unchanged.\\
Moreover, the result in \cite{RoSe17} is stated for viscosity solutions to fully nonlinear problems, instead of distributional solutions to a linear equation, and for kernels that are pointwise comparable to the one of the fractional Laplacian, instead of kernels that merely satisfy \eqref{eq:Glower} and \eqref{eq:Gupper}. However, note that the proof of \cite[Proposition 5.4]{RoSe17} itself, but also of the auxiliary result \cite[Lemma 5.3]{RoSe17}, can be rewritten into the setup of our work without changing any of the arguments. In fact, the Liouville theorem in a half-space (see \cite[Theorem 2.6.2]{FeRo23}) and the stability of distributional solutions (see \cite[Proposition 2.2.31]{FeRo23}) also hold true in our setup. The same applies to the H\"older regularity up to the boundary in Lipschitz domains, \cite[Lemma 5.2]{RoSe17}, which is reproved in this article (see \autoref{lemma:reg-up-to-Lipschitz}). 
\end{proof}

Now, we are in the position to explain how \autoref{lemma:4.4.13} and \autoref{lemma:proposition5.4} can be used to deduce $C^{1,\gamma}$-regularity of the free boundary:


The following proposition directly implies \autoref{prop:4.4.15}:

\begin{proposition}
\label{prop:4.4.15-prelim}
Assume \eqref{eq:Glower} and \eqref{eq:Gupper}. Let $K$ be homogeneous, $\alpha \in (0,\min \{ s , 1-s\} )$ and $\kappa_0 > 0$. Then, there are $\eps > 0$, $R_0 > 1$, depending only on $n,s,\lambda,\Lambda,\alpha,\kappa_0$, such that the following holds true:\\
Let $u \in C^{0,1}(\R^n)$ be such that
\begin{itemize}
\item[(i)] $\min \{ Lu - f , u \} = 0$ in $B_{R_0}$ in the distributional sense, where $|\nabla f| \le \eps$,
\item[(ii)] $D^2 u \ge - \eps \mathrm{Id}$ in $B_{R_0}$ with $0 \in \partial\{ u > 0 \}$,
\item[(iii)] $\Vert \nabla u \Vert_{L^{\infty}(B_R)} \le R^{s+\alpha}$ for any $R \ge R_0$,
\item[(iv)] $\Vert u - \kappa (x \cdot e)_+^{1+s} \Vert_{C^{0,1}(B_{R_0})} \le \eps$ for some $\kappa \ge \kappa_0$ and $e \in \mathbb{S}^{n-1}$.
\end{itemize}
Then, the free boundary $\partial \{ u > 0 \}$ is a $C^{1,\gamma}$-graph in $B_{1/2}$ and moreover, we have
\[\begin{split}
\Vert \nabla u \Vert_{C^s(B_{1/2})} \le C, \qquad \Vert \nabla u / d^s \Vert_{C^{\gamma}(\{u > 0\} \cap B_{1/2})} \le C,
\end{split}\]
where $C > 0$ depends only on $n,s,\lambda,\Lambda,\alpha,\kappa_0$, and $\gamma > 0$ depends only on $n,s,\lambda,\Lambda,\kappa_0$.
\end{proposition}

\begin{proof}
We split the proof into several steps.\\
\textbf{Step 1:}
Let $\kappa_0 > 0$ be as in \autoref{prop:4.4.15}, $\delta_0 > 0$ be as in \autoref{lemma:proposition5.4} and $\delta' > 0$ be as in \autoref{lemma:4.4.13} and $\rho > 1$ be as in \autoref{lemma:lower-dist-bound}. Let us take $\delta < \min\{ \delta_0 , \rho \}$ and $\rho_0 > \rho$. Moreover, recall \eqref{eq:three-and-four}. Then, by \autoref{lemma:4.4.13}, we can find $\eps > 0$ so small and  $R_0 > 1$ so large that
\[\begin{split}
\{ u > 0 \} \cap B_{\rho_0} ~~ \text{ is Lipschitz with constant less than } \delta.
\end{split}\]
Moreover, by \autoref{lemma:lower-dist-bound}, we can find $\eps > 0$ so small and  $R_0 > 1$ so large that (after a rotation)
\begin{align}
\label{eq:lower-dist-bound2}
\partial_n u \ge c d^{s+\theta} ~~ \text{ in } B_{\rho_0}
\end{align}
for some $\theta \in (0,\alpha)$.
Finally, note that if we choose $\eps < \delta$ and $R_0 > 1/\delta$, then by assumption we have in the distributional sense:
%
\begin{align}
\label{eq:C1alpha-v1}
|L (\nabla u)| &\le \delta ~~ \text{ in } B_{1/\delta} \cap \{u > 0\},\\
\label{eq:C1alpha-v2}
\nabla u &= 0 ~~ \text{ in } B_{1/\delta} \setminus \{ u > 0\},\\
\label{eq:C1alpha-v3}
\Vert \nabla u \Vert_{L^{\infty}(B_R)} &\le R^{s+\alpha} ~~ \forall R \ge 1.
\end{align}

In particular, we obtain
\[\begin{split}
C_0^{-1} \le \sup_{B_1} \partial_n u \le C_0
\end{split}\]
for some constant $C_0 > 0$. 
Let us choose $\eps > 0$ so small and $R_0 > 1$ so large that all of the previous properties, i.e. \eqref{eq:lower-dist-bound2}, \eqref{eq:C1alpha-v1}, \eqref{eq:C1alpha-v2}, and \eqref{eq:C1alpha-v3} hold true. Then, we can apply \autoref{lemma:proposition5.4} and infer that for any $x_0 \in \partial \{ u > 0\} \cap B_{1/2}$ there exists $K(x_0) \in \R$  with $|K(x_0)| \le C$ such that it holds
\begin{align}
\label{eq:K-Holder-estimate}
|\partial_i u(x) - K(x_0)\partial_n u(x)| \le C |x-x_0|^{s+\alpha} \qquad \forall x \in B_{1}.
\end{align}

\textbf{Step 2:} First, we prove that $ \partial_i u  / \partial_n u \in L^{\infty}(\{ u > 0\} \cap B_{1})$. To see this, let $x \in \{ u > 0\} \cap B_{1}$ and take $x_0 \in \partial \{ u > 0 \} \cap B_1$ such that $|x - x_0| \le 2 d(x)$, where $d := d_{\{ u = 0\}}$. Then, we obtain
\begin{align}
\label{eq:Linfty-help1}
\left| \frac{\partial_i u(x)}{\partial_n u(x)} - K(x_0) \right| \le \frac{|\partial_i u(x) - K(x_0) \partial_n u(x)|}{|\partial_n u(x)|} \le c \frac{|x-x_0|^{s+\alpha}}{d^{s+\theta}(x)} \le c d^{\alpha - \theta}(x),
\end{align}
where we used \eqref{eq:K-Holder-estimate} and \eqref{eq:lower-dist-bound2}. Therefore,
\[\begin{split}
\left| \frac{\partial_i u(x)}{\partial_n u(x)} \right| &\le |K(x_0)| + \left| \frac{\partial_i u(x)}{\partial_n u(x)} - K(x_0) \right| \le C + cd^{\alpha - \theta}(x) \le c,
\end{split}\]
as desired.

\textbf{Step 3:} We are now in the position to prove that $\partial_i u  / \partial_n u \in C^{\alpha /4}(\{ u > 0\} \cap B_{1})$. We take $x,y \in \{ u > 0 \} \cap B_1$ and assume without loss of generality that $d(y) \le d(x)$. We distinguish between two different cases.\\
Case 1: First, we assume that $d(y) \le d(x) \le 2|x-y|$. In that case, given $x_0, y_0 \in \partial \{ u > 0 \} \cap B_1$ such that $d(x) \le 2 |x-x_0|$ and $d(y) \le 2 |y - y_0|$, we infer from \eqref{eq:Linfty-help1}:
\begin{align}
\label{eq:Holder-help1}
\begin{split}
\left| \frac{\partial_i u(x)}{\partial_n u(x)}  - \frac{\partial_i u(y)}{\partial_n u(y)} \right| &\le \left| \frac{\partial_i u(x)}{\partial_n u(x)} - K(x_0) \right| + \left| \frac{\partial_i u(y)}{\partial_n u(y)} - K(y_0) \right| \\
&\le c(d^{\alpha - \theta}(x) + d^{\alpha - \theta}(y)) \le c |x-y|^{\alpha - \theta}.
\end{split}
\end{align}
Second, we assume that $d(x) \ge 2 |x-y|$. In particular, this means that $y \in B_{d(x)/2}(x)$. Therefore, by \eqref{eq:K-Holder-estimate}, taking $x_0 \in \partial\{ u > 0\} \cap B_1$ with $|x-x_0| \le 2 d(x)$, there exists $K := K(x_0) \in \R$ such that $|K| \le C$ and
\[\begin{split}
\Vert \partial_i u - K \partial_n u \Vert_{L^{\infty}(B_{d(x)/2}(x))} \le c d^{s+\alpha}(x).
\end{split}\]
Therefore, using interior regularity estimates (see \cite[Theorem 2.4.3]{FeRo23}), we obtain
\[\begin{split}
[\partial_i u - K \partial_n u]_{C^{\alpha-\theta}(B_{d(x)/2}(x))}  \le c d^{s+\theta}(x).
\end{split}\]
Indeed,
\[\begin{split}
 & |[\partial_i u  - K \partial_n u](x) - [\partial_i u - K \partial_n u](y)| \\
&\le c \frac{|x-y|^{\alpha - \theta}}{d^{\alpha - \theta}(x)}  \Bigg( \Vert \partial_i u - K \partial_n u \Vert_{L^{\infty}(B_{d(x)/2}(x))} \\
&\qquad \qquad \qquad \qquad + d^{s+\alpha}(x) \sup_{R \ge d(x)/2} \frac{\Vert \partial_i u \Vert_{L^{\infty}(B_R)} + \Vert \partial_n u \Vert_{L^{\infty}(B_R)}}{R^{s+\alpha}}  + d^{2s}(x) \Vert \nabla f \Vert_{L^{\infty}(B_{d(x)/2}(x))} \Bigg)\\
&\le c \frac{|x-y|^{\alpha - \theta}}{d^{\alpha - \theta}(x)} d^{s+\alpha}(x) \le c d^{s+\theta}(x) |x-y|^{\alpha - \theta},
\end{split}\]
where we made use of \eqref{eq:C1alpha-v1} and \eqref{eq:C1alpha-v3}.

By combination of the previous two estimates, using also \eqref{eq:lower-dist-bound2} and that $\Vert \partial_n u \Vert_{C^{s-\tilde{\eps}}}(B_1) \le c$ for any $\tilde{\eps} \in (0,1)$ due to \autoref{remark:almost-opt}, we obtain:
\[\begin{split}
\left| \frac{\partial_i u(x)}{\partial_n u(x)} - \frac{\partial_i u(y)}{\partial_n u(y)} \right| &= \left| \frac{[\partial_i u - K \partial_n u](x)}{\partial_n u(x)} - \frac{[\partial_i u - K \partial_n u](y)}{\partial_n u(y)} \right| \\
&\le \left|\frac{1}{\partial_n u(y)} + \frac{1}{\partial_n u(x)} \right| \left| [\partial_i u - K \partial_n u](x) - [\partial_i u - K \partial_n u](y)\right| \\
&+ \left| \frac{[\partial_i u - K \partial_n u](x) + [\partial_i u - K \partial_n u](y)}{\partial_n u(x) \partial_n u(x)} \right| |\partial_n u(x) - \partial_n u(y)| \\
&\le c \frac{d^{s+\theta}(x)}{d^{s+\theta}(x) + d^{s+\theta}(y)} |x-y|^{\alpha - \theta} + c \frac{d^{s+\alpha}(x)}{d^{s+\theta}(x)d^{s+\theta}(y)} |x-y|^{s-\tilde{\eps}}\\
&\le c |x-y|^{\alpha - \theta} + d^{-s+\alpha - 2 \theta}(x) |x-y|^{s-\tilde{\eps}}\\
&\le c(|x-y|^{\alpha - \theta} + |x-y|^{\alpha/4}),
\end{split}\]
where we used in the last step that $d(y) \ge d(x) /2$, and also $\alpha - 2 \theta > \alpha/2$ and we chose $\tilde{\eps} < \alpha/4$ so that 
\[\begin{split}
d^{-s+\alpha - 2 \theta}(x) |x-y|^{s - \tilde{\eps}} \le d^{-s + \frac{\alpha}{2}}(x) |x-y|^{s - \tilde{\eps}} \le |x-y|^{\frac{\alpha}{2} - \tilde{\eps}} \le |x-y|^{\alpha/4}.
\end{split}\]
After combination with Step 2 and \eqref{eq:Holder-help1} we obtain  that $\partial_i u  / \partial_n u \in C^{\alpha/4}(\{ u > 0\} \cap B_{1})$.

\textbf{Step 4:} We have shown that for some $\gamma \in (0,1)$:
\begin{align}
\label{eq:quotient-Holder}
\left\Vert \partial_i u  / \partial_n u \right\Vert_{C^{\gamma}(\{ u > 0\} \cap B_{1})} \le C.
\end{align}

From here, the same arguments as in \cite{CRS17} yield that the free boundary is $C^{1,\gamma}$. In fact, the normal vectors of the level set $\{ u = t\}$ are of the form
\[\begin{split}
\nu_i(x) = \frac{\frac{\partial_i u(x)}{\partial_n u(x)}}{\left( \sum_{j = 1}^{n-1} \left(\frac{\partial_j u(x)}{\partial_n u(x)}\right)^2  + 1\right)^{1/2}} \qquad \forall x \in \{ u = t\}.
\end{split}\]
Hence, by \eqref{eq:quotient-Holder}, $\Vert \nu \Vert_{C^{\gamma}(\{ u = t\} \cap B_1)} \le C$. Thus, the level sets $\{ u  = t \}$ are uniformly $C^{1,\gamma}$ in $B_1$. Therefore, they converge (resp. their graphs converge) uniformly as $t \searrow 0$ to $\partial \{ u > 0\}$, which implies that $\partial \{u > 0\}$ is a $C^{1,\gamma}$-graph in $B_1$.\\
Next, we recall \cite[Corollary 1.3, using that $\gamma = s$ in the symmetric case]{DRSV22}), which states that due to the $C^{1,\gamma}$-regularity of the free boundary,
\[\begin{split}
\Vert \nabla u / d^s \Vert_{C^{\gamma_0}(\{u > 0\} \cap B_{1/2})} \le C
\end{split}\]
for some $\gamma_0 \in (0,\gamma]$ depending only on $n,s,\lambda,\Lambda,\gamma$, and $C > 0$, depending only on $n,s,\lambda,\Lambda,\alpha,\gamma$. Note that for kernels satisfying the upper bound in \eqref{eq:Kcomp}, this result can be found in \cite[Proposition 2.6.8]{FeRo23}.\\
Moreover, note that by \cite[Proposition 2.5.4]{FeRo23}, we have
\[\begin{split}
\Vert \nabla u \Vert_{C^s(B_{1/2})} \le C
\end{split}\]
for some $C > 0$, depending only on $n,s,\lambda,\Lambda,\alpha,\kappa_0$. Finally, observe that for the application of both results, we used that $\nabla u \in C(\overline{B_{1}})$ by \autoref{prop:C1-tau} (see also \cite[Remark 2.5.5]{FeRo23}).
\end{proof} 

\begin{proof}[Proof of \autoref{prop:4.4.15}]
Let us apply \autoref{prop:4.4.15-prelim} with $v(x) = R^{1+s} u(x/R)$ for some $R > 0$. Then, it holds
\[\begin{split}
\min \{ Lv - \tilde{f} , v \} &= 0 ~~ \text{ in } B_{R}, ~~ \text{ where } |\nabla \tilde{f}| \le R^{-s},\\
D^2 v &\ge - R^{s-1} ~~ \text{ in } B_R,\\
\Vert \nabla v \Vert_{L^{\infty}(\R^n)} &\le R^s,\\
\Vert v - \kappa (x \cdot e)_+^{1+s} \Vert_{C^{0,1}(B_{R})} &\le \eps.
\end{split}\]
Next, we choose $R \ge R_0$ so large that $R^{-s} \le R_0$, $R^{s-1} \le R_0$
The desired result follows upon application of \autoref{prop:4.4.15-prelim} to $v$ and rescaling back to $u$.
\end{proof}

\subsection{Proof of \autoref{thm:opt-reg}}

We finally give the:

\begin{proof}[Proof of \autoref{thm:opt-reg}]
Note that by \autoref{lemma:semiconvexity}, $u \in C^{0,1}(\R^n)$. Let us now denote $v = u - \phi$. Then, by \cite[Lemma 2.2.27]{FeRo23}, $v$ is also a distributional solution to $\min \{ L v - f , v\} = 0$ in $\R^n$. From here, the proof of \autoref{thm:opt-reg} follows by combination of \autoref{lemma:prop4.4.14} and \autoref{prop:4.4.15-prelim} just as in the proof of \cite[Corollary 2.16]{FRS23} or \cite[Proposition 4.5.2]{FeRo23}.
\end{proof}


\end{document}